%% file: main.tex
\RequirePackage{fix-cm}

\documentclass[smallextended]{svjour3}       

\smartqed  

\usepackage{color}

\usepackage[left=1.0in,right=1.0in,top=1in,bottom=1in]{geometry}

\usepackage{graphicx}
\usepackage{amsmath,amssymb}
\usepackage{mathrsfs}
\usepackage{amssymb}
\usepackage{subfigure}
\usepackage{bm}
\usepackage{amsfonts}
\usepackage{amsxtra}
\usepackage{upgreek}
\usepackage{mathtools}
\usepackage{enumerate}
\usepackage{framed}
\usepackage{indentfirst}
\usepackage{pstricks}
\usepackage{tikz}
\usepackage{cite}
\usepackage{lipsum}
\usepackage[scr=boondoxo]{mathalfa}

\DeclareMathOperator*{\argmin}{arg\,min}

\usepackage[misc]{ifsym}

\usetikzlibrary{patterns}
\usepackage[colorlinks=true,linkcolor=blue]{hyperref}
\graphicspath{ {figures/} }
\begin{document}
\newcommand{\COMMENT}[1]{\textcolor{red}{#1}}
\newcommand{\INSERT}[1]{\textcolor{blue}{#1}}
\renewcommand{\thefootnote}{\arabic{footnote}}
\title{HiDeNN-PGD: reduced-order hierarchical deep learning neural networks }
\newcommand\blfootnote[1]{%
\begingroup{
\renewcommand\thefootnote{}\footnote{#1}%
\addtocounter{footnote}{-1}%
}
\endgroup
}
\author{
Lei Zhang\textsuperscript{1,3}\and
Ye Lu\textsuperscript{2}\and
Shaoqiang Tang\textsuperscript{1*}\and
Wing Kam Liu\textsuperscript{2*} 
}

\institute{\textsuperscript{1} HEDPS and LTCS, College of Engineering, Peking University, Beijing 100871, China.
\and
\textsuperscript{2} Department of Mechanical Engineering, Northwestern University,  Evanston, USA. 
\textsuperscript{3} Visiting student at Department of Mechanical Engineering, Northwestern University
\and
\textsuperscript{*} Corresponding Author: w-liu@northwestern.edu, maotang@pku.edu.cn
}

\date{Received: date / Accepted: date}

\maketitle
\begin{large}
\begin{abstract}
This paper presents a proper generalized decomposition (PGD) based reduced-order model of hierarchical deep-learning neural networks (HiDeNN). The proposed HiDeNN-PGD method keeps both advantages of HiDeNN and PGD methods. The automatic mesh adaptivity makes the HiDeNN-PGD more accurate than the finite element method (FEM) and conventional PGD, using a fraction of the FEM degrees of freedom. The accuracy and convergence of the method have been studied theoretically and numerically, with a comparison to different methods, including FEM, PGD, HiDeNN and Deep Neural Networks. In addition, we theoretically showed that the PGD converges to FEM at increasing modes, and the PGD error is a direct sum of the FEM error and the mode reduction error.  The proposed HiDeNN-PGD performs high accuracy with orders of magnitude fewer degrees of freedom, which shows a high potential to achieve fast computations with a high level of accuracy for large-size engineering problems.
\end{abstract}


\section{Introduction}
Despite the constantly increasing computer power, numerical simulations of physical systems with numerous degrees of freedom remain computationally prohibitive. These kinds of problems arise usually in simulation-based engineering applications and the repetitive manipulation (or modification) of the mesh system has been identified as a key time-costly issue in standard finite element method (FEM) \cite{belytschko2013nonlinear}. This has been a motivation for developing the isogeometric approaches \cite{hughes2005isogeometric}. 

In recent years, the deep neural network (DNN) has shown some interesting features in handling the solution of physics constrained systems. The universal approximation theorem \cite{hornik1989multilayer,cybenko1989approximation} and the natural scalability of DNN have been the foundation of its superior performance for large systems. This has thus motivated the use of DNN to approximate the solution of partial differential equations (PDEs) \cite{weinan2017deep,raissi2019physics,li2019clustering}. A recently developed Hierarchical Deep-learning Neural Network (HiDeNN) method \cite{zhang2021hierarchical,saha2021hierarchical} falls within this perspective. The so-called HiDeNN is developed by constraining the weights and biases of DNN to mesh coordinates to build multiple dimensions finite element, meshfree, isogeometric, B-spline, and NURBs interpolation functions. HiDeNN allowed the automatic mesh adaptivity and showed a good potential to prevent large mesh systems and the standard time-consuming mesh refinement procedure.
In order to further enhance the efficiency of HiDeNN, this work proposed HiDeNN-PGD, a reduced-order model of HiDeNN using the proper generalized decomposition (PGD).

The PGD-based model reduction methods rely on the idea of separation of variables, and are usually written in the format of canonical decomposition. This kind of method was originally proposed in an \textit{a priori} setting \cite{ammar2006new,chinesta2011short,bhattacharyya2018multi,modesto2015proper}, in which the separated functions are computed on-the-fly by solving the PDEs. It has gained increased popularity in recent years. For overcoming the intrusiveness and extending the applicability of the method, the \textit{a posteriori} data-driven PGD \cite{lu2018adaptive,lu2019datadriven,lu2018multi,diez2018algebraic}  has also been developed more recently. In contrast to \textit{a priori} PGD, the \textit{a posteriori} method uses a  database to learn the separated functions and thus can be used as regression for constructing reduced order surrogate models. 

In our work, we adopted the same idea of separation of variables, in particular, the separation of space variables for solving PDEs in the HiDeNN framework, leading to the so-called HiDeNN-PGD method, which is expected to have reduced degrees of freedom with a high accuracy. Indeed, the space separated PGD, leading to lower dimensional space functions, is usually considered for reducing the computational complexity of 3D separable domains (see e.g. \cite{bognet2014separated}).  However, the convergence aspect with respect to the mesh refinement and number of modes has been less studied. 

We investigated the convergence aspect of the PGD approach in this paper. Based on the approximation function spaces, we  analyzed the numerical error and convergence associated with different approaches and compare their error bounds. It can be shown that the HiDeNN-PGD is more accurate than both FEM and conventional PGD, thanks to the adaptivity achieved by HiDeNN. For further enhancing the optimality of PGD modes, we suggested fixing the number of modes firstly and solve them together. Hence, HiDeNN-PGD can require fewer modes than PGD. This is advantageous for high-dimensional problems where the optimality of modes is crucial. The numerical examples have confirmed our theoretical analysis. In addition, we numerically investigated the relationship between the approximation error and the modes and proposed a strategy to select the prescribed mode number in HiDeNN-PGD. The proposed HiDeNN-PGD has shown a high potential to achieve high performance computing with high accuracy.

The paper is organized as follows. Section 2 gives  a brief overview of different numerical methods for partial differential equations (PDEs), in which the approximation function spaces are described. The error analysis based on a class of PDEs is given in Section 3. Section 4 presents the proposed HiDeNN-PGD method. Section 5 provides some numerical examples and discussions. Finally, the paper closes with some concluding remarks.

\input{Part1_ReviewAndAnalysis}

\input{Part2_HiDeNNPGD}

\input{NumericalExamples}

\section{Conclusion}
A reduced-order hierarchical deep learning network has been proposed. The so-called HiDeNN-PGD is a combination of HiDeNN and PGD with separated spatial variables. This combined method presents several advantages over HiDeNN and PGD methods. First, it allows to leverage the automatic mesh adaptivity of the HiDeNN method for reducing the mode number in PGD approximation. Second, combining PGD with HiDeNN reduces significantly the number of degrees of freedom for HiDeNN and potentially leads to a high computational efficiency.  Furthermore, we have demonstrated that both HiDeNN and HiDeNN-PGD can provide more accurate solutions than FEM and PGD (or CD), through an error analysis with the help of analyzing the approximation function spaces. 

The numerical results have confirmed the mathematical analysis. These examples have been performed based on 2D and 3D Poisson problems. It is shown that the proposed HiDeNN-PGD method can provide accurate solutions with the least degrees of freedom. In order to have an idea for the prescribed number of modes in HiDeNN-PGD, we have studied numerically the convergence rate on PGD approximation. It has been found that the convergence rate on the mode number is insensitive to the mesh size. Therefore, we can expect to use a coarse mesh PGD to compute a roughly estimated mode number for HiDeNN-PGD. This finding is interesting and provides a useful guideline on the choice of the number of modes for HiDeNN-PGD or other PGD-based methods that may require a better optimality in terms of basis.

In the future,  the computational efficiency of HiDeNN-PGD will be further explored based on  realistic problems, such as thermo-mechanical analysis in additive manufacturing or multi-scale composite simulations. In terms of convergence studies, theoretical results need to be derived through a rigorous mathematical analysis. The numerical results provided in this paper can serve as the first evidence for demonstrating the capabilities of the method. 

\section*{Acknowledgement}
L. Zhang, and S. Tang are supported by National Natural Science Foundation of China grant number 11890681, 11832001, 11521202 and 11988102.
W. K. Liu and Y. Lu are supported by National Science Foundation grant number CMMI-1934367 and CMMI-1762035.

\appendix

\input{Appendix}

\input{main.bbl}
\bibliographystyle{elsarticle-num.bst}

\nocite{*}

\end{large}
\end{document}

%% file: Part1_ReviewAndAnalysis.tex
\section{ DNN, HiDeNN, FEM and canonical decomposition  based function approximation}

Function approximation is a key component in numerical solutions of partial differential equations (PDEs). In this section, we briefly review how such an approximation can be performed in terms of FEM, DNN, HiDeNN and Canonical tensor Decomposition (CD). We present their approximation function sets, which will be used in the theoretical analysis in Subsection 2.2.
We restrict the discussion to a scalar-value function ${u}(\bm{x}):\mathbb{R}^3\rightarrow \mathbb{R}$. The conclusions should be straightforwardly extended to vector functions.

\begin{figure}[h]
	\centering
	\subfigure[DNN-based interpolation function]{\includegraphics[trim={1.2cm 0 0 0},width=0.41\textwidth]{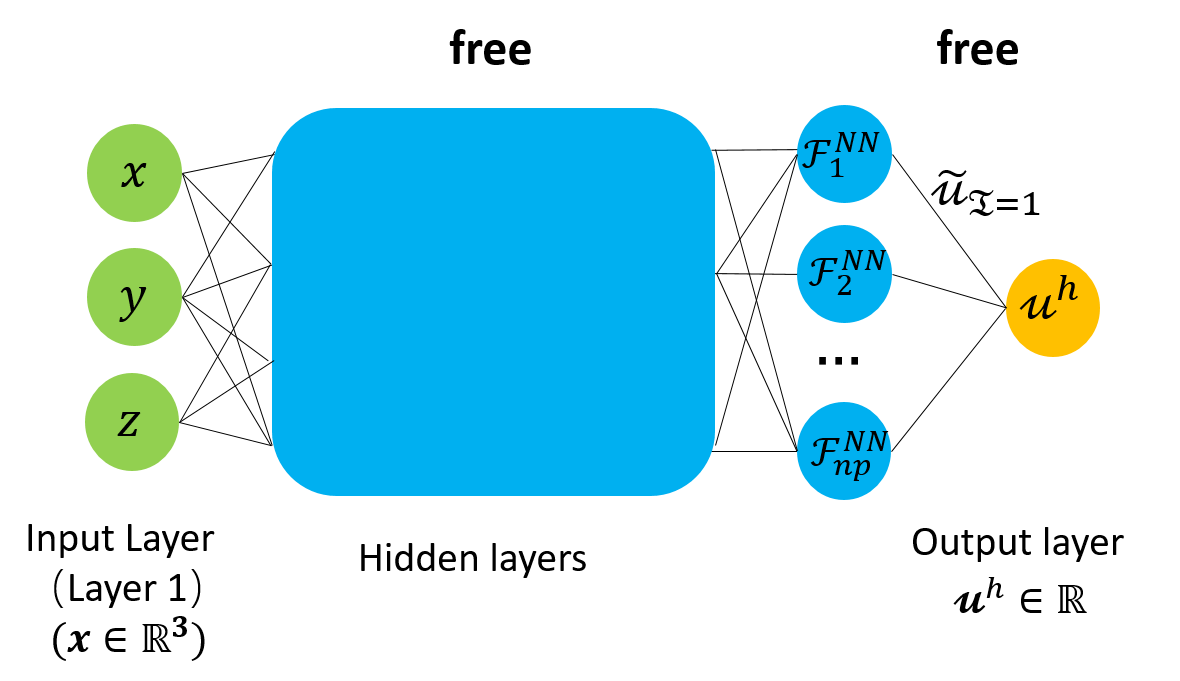}}
	\subfigure[HiDeNN interpolation function]{\includegraphics[width=0.45\textwidth]{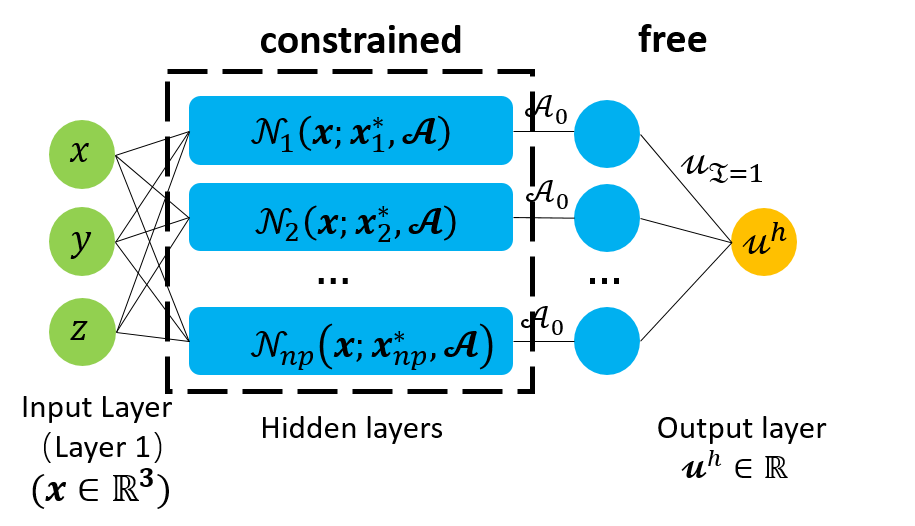}}
	\subfigure[FEM interpolation function]{\includegraphics[width=0.45\textwidth]{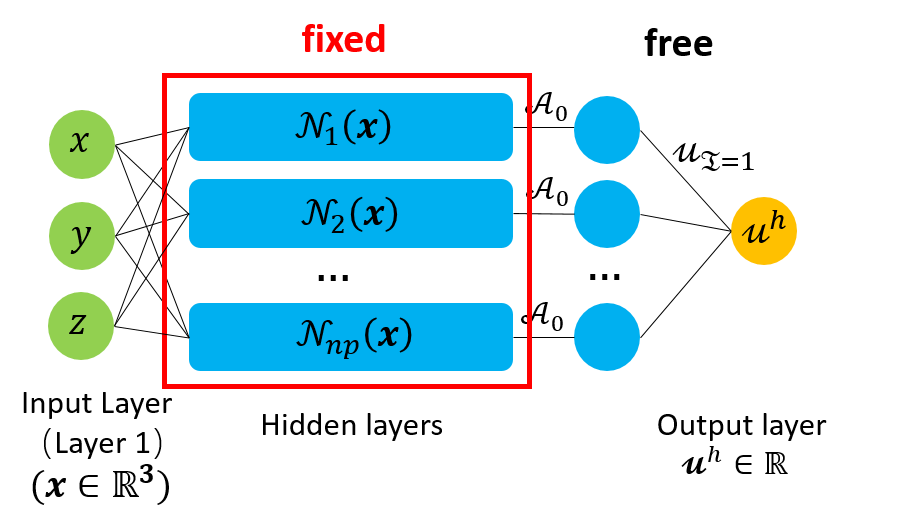}}
	\subfigure[CD interpolation function]{\includegraphics[width=0.45\textwidth]{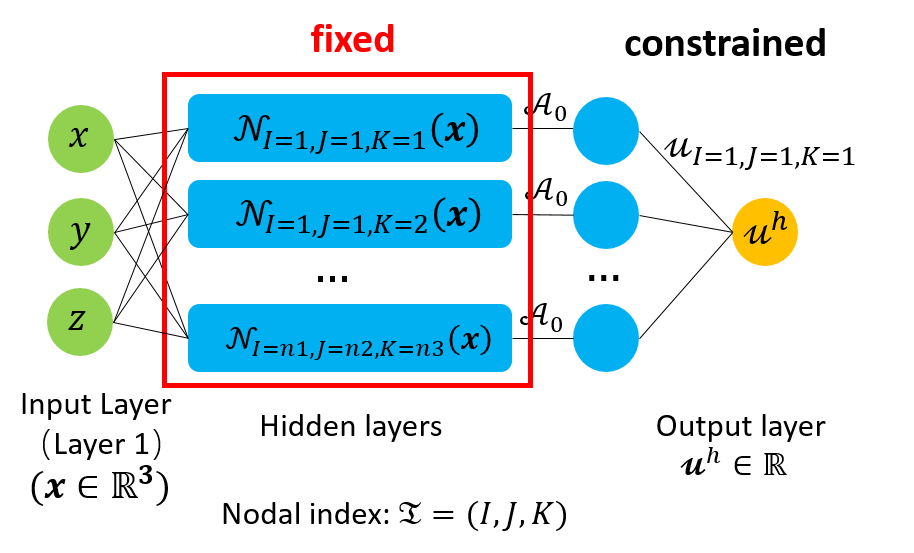}}
	\caption{Illustration for the interpolation functions in the form of DNN with $\bm{x}=(x,y,z)$ as input and $u^h$ as output. Weights and biases inside dashed line-box are constrained, and those inside the red solid line-box are fixed. $\mathcal{A}_0$ is the identity activation function defined by $\mathcal{A}_0(x)=x$. }
	\label{fig:Interpolant_DNNrepresentation}
\end{figure}

\subsection{Overview of the approximation function sets}
\textbf{Deep neural network-based method}

According to the universal approximation theorem, a deep neural network (DNN) can be designed to approximate any given continuous function to desired accuracy \cite{hornik1989multilayer,cybenko1989approximation,tang2021neural}. Thus it can be a candidate to  approximate solutions for solving PDEs \cite{wu2016physics,xiao2016quantifying,weinan2017deep,berg2018unified,raissi2017physics,raissi2019physics,sirignano2018dgm,weinan2018deep}, i.e.,
\begin{equation}
    u^h(\bm{x})=\mathcal{F}^{NN}(\bm{x}),
\end{equation}
where $\mathcal{F}^{NN}$ represents the neural network with $\bm{x}$ as input and $u^h$ as output. Note that $u^h$ can be a multidimensional vector.
For instance, in a classical feedforward neural network (FFNN) \cite{Goodfellow-et-al-2016,haykin1994neural,oishi2017computational} with $N_L$ layers, recursive relations among neurons are as follows
\begin{eqnarray} \label{eq:transfer-function}
&&a^{l}_{j=1}=x, a^{l}_{j=2}=y, a^{l}_{j=3}=z, \text{ if } l=1 \text{ (input layer)}; \\
&&a^{l}_j=\mathcal{A}(\sum_{i=1}^{N_N(l-1)}{W^{l}_{ij} a^{l-1}_i + b^{l}_j}), \text{ if } l\in\{2,...,N_L-1\}  \text{ (hidden layer)}.
\end{eqnarray}
Hence, the output layer can be defined as
\begin{eqnarray}
\mathcal{F}^{NN}_j&=&a^{N_L}_j=\sum_{i=1}^{N_N(N_L-1)}{W^{N_L}_{ij} a^{N_L-1}_i + b^{N_L}_j}, \text{ if } l=N_L \text{ (output layer)},
\end{eqnarray}
with the detailed definition of the notations in Table \ref{table:ffnn-notation}. Therefore, once the weights $\bm{W}$, biases $\bm{b}$ and activation functions $\bm{\mathcal{A}}$ have been chosen, $\mathcal{F}^{NN}$ can serve as an approximation function with the input variable as $\bm{x}=(x,y,z)$. 

\begin{table}[h!]
\centering
\caption{Notation table of variables used in the feed forward neural network} 
\label{table:ffnn-notation}
\begin{tabular}{cl}
  \hline
      $\bm{x}=(x,y,z)$ & Space coordinates \\
      $l$ & Counting index for number of layers \\
      $i$ & Counting index for neurons in layer $l-1$ \\
      $j$ & Counting index for neurons in layer $l$ \\
      $N_L$ & Number of layers in the neural network \\
      $N_N(l)$ & Number of neurons in layer $l$ \\
      $W_{ij}^{l}$ & Weight connecting the $i^\text{th}$ neuron in layer $l-1$ to the $j^\text{th}$ in layer $l$ \\
      $b_{j}^{l}$ & Bias of the $j^\text{th}$ neuron in layer $l$\\
      $a_j^{l}$ & Neuron value for $j^{th}$ neuron in $l^{th}$ layer \\
      $\mathcal{A}$ & Activation function \\
      $\mathcal{F}^{NN}$ & Feedforward neural network function\\
 \hline
\end{tabular}
\end{table}

The approximation function set formed by a general NN is
\begin{equation}
    \mathcal{N}^h=\left\{ u^h(\bm{x}) \big | u^h=\mathcal{F}^{NN}(\bm{x};\bm{W},\bm{b},\bm{\mathcal{A}}), W^l_{ij}\in \mathbb{R}, b^l_j\in \mathbb{R} \right\},
\end{equation}
where $\mathcal{F}^{NN}(\bm{x};\bm{W},\bm{b},\bm{\mathcal{A}})$ denotes an NN with the input $\bm{x}$ and depends on weights $\bm{W}$, biases $\bm{b}$ and activation functions $\bm{\mathcal{A}}$.

For interpretation, DNN can be somehow  rewritten in the form of shape functions associated to nodal values of $ u^h$. In this way, the function approximation reads
\begin{equation}
    u^h=\sum_{\mathcal{I}=1}^{np} \mathcal{F}^{NN}_\mathcal{I}(\bm{x};\bm{W},\bm{b},\bm{\mathcal{A}}) \mathscr{u}_{\mathcal{I}},
    \label{eq:InterpolationNN}
\end{equation}
as illustrated in Fig. \ref{fig:Interpolant_DNNrepresentation}(a). $\mathcal{F}^{NN}_\mathcal{I}$ represents the value of the $\mathcal{I}$-th neuron in the last hidden layer, i.e., the output of the previous hidden layers. $\mathscr{u}_{\mathcal{I}}$ is the corresponding weight connecting the output layer with the last hidden layer.
This interpolation form may provide a more interpretable structure for DNN. 

In multidimensional cases, such as for 3D mechanical problems, the above equation can be straightforwardly applied to each component of displacement fied as follows
\begin{equation}
    u_x^h=\sum_{\mathcal{I}=1}^{np} \mathcal{F}^{NN}_\mathcal{I}(\bm{x};\bm{W},\bm{b},\bm{\mathcal{A}}) \mathscr{u}_{x\mathcal{I}},\\
\end{equation}
\begin{equation}
    u_y^h=\sum_{\mathcal{I}=1}^{np} \mathcal{F}^{NN}_\mathcal{I}(\bm{x};\bm{W},\bm{b},\bm{\mathcal{A}}) \mathscr{u}_{y\mathcal{I}},
\end{equation}
\begin{equation}
    u_z^h=\sum_{\mathcal{I}=1}^{np} \mathcal{F}^{NN}_\mathcal{I}(\bm{x};\bm{W},\bm{b},\bm{\mathcal{A}}) \mathscr{u}_{z\mathcal{I}},
\end{equation}




\textbf{HiDeNN}

 The recently developed HiDeNN method \cite{zhang2021hierarchical} uses a similar DNN structure of \eqref{eq:InterpolationNN} with additional constraints to build a family of function approximations. Similar to the  FEM, the continuous domain
$\Omega$ is discretized by a mesh with $np$ nodes $\bm{x}_1,\bm{x}_2,\cdots,\bm{x}_{np}$. Then the finite element shape functions $N_{\mathcal{I}}$ can be constructed by the neural network block, namely, 
\begin{equation}
     \mathcal{N}_{\mathcal{I}}(\bm{x};\bm{W},\bm{b},\bm{\mathcal{A}})
\end{equation}
where $\mathcal{I}=1,2,\cdots,np$. Different from $\mathcal{F}_{\mathcal{I}}^{NN}(\bm{x};\bm{W},\bm{b},\bm{\mathcal{A}})$ in (\ref{eq:InterpolationNN}), $\mathcal{N}_{\mathcal{I}}(\bm{x};\bm{W},\bm{b},\bm{\mathcal{A}})$ precisely equals to the finite element shape function $N_{\mathcal{I}}(\bm{x})$ with inputs $\bm{x}$ and an output $N_{\mathcal{I}}$, satisfying the following constraints for  shape functions automatically, i.e.
\begin{eqnarray}
    \sum_{\mathcal{I}=1}^{np} \mathcal{N}_{\mathcal{I}}(\bm{x};\bm{W},\bm{b},\bm{\mathcal{A}})=1, \mathcal{N}_{\mathcal{I}}(\bm{x}_{\mathcal{J}};\bm{W},\bm{b},\bm{\mathcal{A}})=\delta_{\mathcal{I}\mathcal{J}}.
    \label{eq:Constraint}
\end{eqnarray}
With Kronecker Delta constraints, we can apply Dirichlet boundary conditions directly similar to that of the finite element method, 
so that all the weights $\bm{W}$,  and biases $\bm{b}$ are functions of nodal coordinates $\bm{x}_I$. Thus we can rewrite the shape function explicitly in terms of $\bm{x}_I$ as
\begin{equation}
    \mathcal{N}_{\mathcal{I}}(\bm{x};\bm{x}_{\mathcal{I}}^*,\bm{\mathcal{A}}) = \mathcal{N}_{\mathcal{I}}(\bm{x};\bm{W},\bm{b},\bm{\mathcal{A}}),
\end{equation}
where $\bm{x}_{\mathcal{I}}^*$ denotes the support of $N_{\mathcal{I}}(\bm{x})$, e.g. in linear 1D cases $\bm{x}_{\mathcal{I}}^*=[x_{{\mathcal{I}}-1},x_{\mathcal{I}},x_{{\mathcal{I}}+1}]$.

Combining such neural network blocks for the entire mesh gives the final form of HiDeNN, as shown in Fig. \ref{fig:1D_DNN_representation}(b). This results in the approximation function set
\begin{equation}
    \mathcal{H}^h=\left\{ u^h(\bm{x}) \bigg | u^h=\sum_{\mathcal{I}=1}^{np} \mathcal{N}_{\mathcal{I}}(\bm{x};\bm{x}_{\mathcal{I}}^*,\bm{\mathcal{A}})\mathscr{u}_{\mathcal{I}}, \mathscr{u}_{\mathcal{I}}\in\mathbb{R} \right\}.
\end{equation}

The parametric expression with nodal positions $\bm{x}_{\mathcal{I}}^*$ allows automatic r-adaptivity, and accordingly improves the local and global accuracy of the interpolant.

\textbf{Finite element method}

The approximation function set $\mathcal{H}^h$ degenerates to the FE approximation function set $\mathcal{V}^h$ \cite{belytschko2013nonlinear} when the nodal position is fixed, which reads
\begin{equation}
    \mathcal{V}^h=\left\{u^h(\bm{x}) \bigg | u^h=\sum_{\mathcal{I}=1}^{np} N_{\mathcal{I}}(\bm{x}) u_{\mathcal{I}}, u_{\mathcal{I}}\in\mathbb{R} \right\}.
\end{equation}

From the DNN viewpoint, this corresponds to fixing the weights and biases that are functions of $\bm{x}_{\mathcal{I}}^*$, as shown in Fig. \ref{fig:Interpolant_DNNrepresentation}(b). 


\textbf{Canonical tensor decomposition}

Under the assumption of separation of variables, the function $\bm{u}$ may be approximated
by the sum of the products of multiple 1D functions, i.e.,
\begin{equation}
    u^h(\bm{x})=u^h(x,y,z)=\sum_{q=1}^Q X^{(q)}(x)Y^{(q)}(y)Z^{(q)}(z),
    \label{eq:SeparaRepresentation}
\end{equation}
where $Q$ is the number of modes, and the product of $X^{(q)}, Y^{(q)},Z^{(q)}$ provides a mode for the interpolation function.
This form or concept is known as CD \cite{kolda2009tensor}. The so-called PGD method has adopted this concept for solving PDEs \cite{ammar2006new,gonzalez2010recent} and for data learning \cite{lu2018adaptive,lu2019datadriven,blal2019non}. 

Thanks to the separation of variables, only shape functions of reduced dimension are needed. In (\ref{eq:SeparaRepresentation}),
1D FE shape functions can be used for a 3D problem, namely,
\begin{eqnarray}
    \label{eq:MSDisX}
    X^{(q)}(x)&=&\sum_{I=1}^{n_1} N_I(x)\beta_I^{(q)}, \\
    \label{eq:MSDisY}
    Y^{(q)}(y)&=&\sum_{J=1}^{n_2} N_J(y)\gamma_J^{(q)}, \\
    \label{eq:MSDisZ}
    Z^{(q)}(z)&=&\sum_{K=1}^{n_3} N_K(z)\theta_K^{(q)}.
\end{eqnarray}
Here, $n_1, n_2, n_3$ are the number of nodes in $x,y,z$ directions, respectively.
Thus the corresponding approximation function set for a given $Q$ modes is
\begin{eqnarray}
    \mathcal{M}_Q^h=&\Bigg \{ &u^h(\bm{x}) \bigg |u^h=\sum_{q=1}^Q \left( \sum_{I=1}^{n_1} N_I(x)\beta_I^{(q)} \right) \left( \sum_{J=1}^{n_2} N_J(y)\gamma_J^{(q)} \right) \left( \sum_{K=1}^{n_3} N_K(z)\theta_K^{(q)} \right), \\ \nonumber 
    &&\beta_I^{(q)},\gamma_J^{(q)},\theta_K^{(q)}\in\mathbb{R} \Bigg \}.
\end{eqnarray}

Note that the interpolation function in (\ref{eq:SeparaRepresentation}) can be rearranged as
\begin{equation}
    u^h(x,y,z)=\sum_{I=1}^{n_1}\sum_{J=1}^{n_2}\sum_{K=1}^{n_3} N_I(x) N_J(y) N_K(z) \left( \sum_{q=1}^Q \beta^{(q)}_I \gamma^{(q)}_J \theta^{(q)}_K \right)
    \label{eq:MSExpand}
\end{equation}
which is regarded as a finite element interpolation function with $N_I(x)N_J(y)N_K(z)$ as shape
functions and $\left( \sum_{q=1}^Q \beta^{(q)}_I \gamma^{(q)}_J \theta^{(q)}_K \right)$ as coefficients. Fig. \ref{fig:Interpolant_DNNrepresentation}(d) illustrates a DNN format of (\ref{eq:MSExpand}), which will be the basis for the proposed HiDeNN-PGD method.

Multidimensional shape functions of CD, i.e., the product of 1D shape functions, are fixed and determined by nodal positions along each direction. In addition, the nodal values in the last layer are constraint in the form of the tensor product, i.e.,
\begin{equation}
    u_{(I,J,K)}=\sum_{q=1}^{Q}\beta^{(q)}_{I}\gamma^{(q)}_J\theta^{(q)}_K.
\end{equation}

When a few modes may represent the function $u^h(x,y,z)$, this method is advantageous in terms of lower integration complexity and less degrees of freedom (DoFs) \cite{ammar2006new,chinesta2013proper}. The DoFs in (\ref{eq:MSExpand}) is of the order $O((n_1+n_2+n_3)Q)$, which is linear with the spatial dimension and far smaller than traditional methods (e.g., FEM). 


\subsection{HiDeNN-PGD: Reduced order HiDeNN via PGD}
\label{ssec:HiDeNN-PGD}

HiDeNN gets better accuracy compared with classical FEM due to the adaptivity of nodal positions. More DoFs might result in more cost. On the other hand, representation of separated variables provides a reduced order model to improve the efficiency but might lose accuracy. Here, we propose HiDeNN-PGD, a reduced-order model of HiDeNN via PGD, which seeks to accomplish an optimized balance between accuracy and computational cost.

In HiDeNN-PGD, the shape functions in each direction are written in the DNN format, namely, (\ref{eq:MSDisX}-\ref{eq:MSDisZ}) are replaced by 1D HiDeNN interpolants (refer to Appendix \ref{sec:1DHiDeNN}),
\begin{eqnarray}
    X^{(q)}(x)&=&\sum_{I=1}^{n_1} \mathcal{N}_I(x;\bm{x}_I^*,\bm{\mathcal{A}})\beta_I^{(q)}, \\
    Y^{(q)}(y)&=&\sum_{J=1}^{n_2} \mathcal{N}_J(y;\bm{y}_J^*,\bm{\mathcal{A}})\gamma_J^{(q)}, \\
    Z^{(q)}(z)&=&\sum_{K=1}^{n_3} \mathcal{N}_K(z;\bm{z}_K^*,\bm{\mathcal{A}})\theta_K^{(q)},
\end{eqnarray}
where $\mathcal{N}_I(x;\bm{x}_I^*,\bm{\mathcal{A}}), \mathcal{N}_J(y;\bm{y}_J^*,\bm{\mathcal{A}}), \mathcal{N}_K(z;\bm{z}_K^*,\bm{\mathcal{A}})$ are the 1D HiDeNN shape functions in $x,y,z$ directions, respectively.
Thus the interpolation function set is defined by
\begin{eqnarray}
\label{eq:HiDeNNPGD_FunSet}
    &&\mathcal{G}_Q^h= \\ \nonumber
    &&\Bigg \{ u^h \bigg |u^h=\sum_{q=1}^Q \left( \sum_{I=1}^{n_1} \mathcal{N}_I(x;\bm{x}_I^*,\bm{\mathcal{A}})\beta_I^{(q)} \right) \left( \sum_{J=1}^{n_2} \mathcal{N}_J(y;\bm{y}_J^*,\bm{\mathcal{A}})\gamma_J^{(q)} \right) \left( \sum_{K=1}^{n_3} \mathcal{N}_K(z;\bm{z}_K^*,\bm{\mathcal{A}})\theta_K^{(q)} \right) \Bigg \}.
\end{eqnarray}
Since adaptivity occurs only in each direction, the mesh is always regular.

\subsection{Relationship among the NN, HiDeNN, FEM, CD and HiDeNN-PGD approximation function sets}

In this subsection, we explore the relationship among FEM, NN, HiDeNN, CD and HiDeNN-PGD. 

Assume that CD, FEM are based on the same regular mesh with $n_1,n_2,n_3$ nodes along $x,y,z$ directions. This mesh also serves as an initial guess of $\bm{x}_I^*$ in HiDeNN and HiDeNN-PGD. The shape functions of FEM is the product of 1D shape functions, i.e., the shape function associated with the node $(x_i,y_j,z_k)$ is $N_{(I,J,K)}(x,y,z)=N_I(x) N_J(y) N_K(z)$. NN has a more general structure than HiDeNN, and might be fully-connected. 

By definition, we have the following relationship among approximation function sets of NN, HiDeNN, FEM and CD:
\begin{equation}
    \mathcal{M}_Q^h \subset \mathcal{V}^h \subset \mathcal{HI}^h \subset \mathcal{N}^h.
    \label{eq:RelationshipAmongMethods}
\end{equation}
Especially when $Q$ is big enough ($Q\geq\min\{n_1,n_2\}$ for 2D and $Q\geq\min\{n_1 n_2,n_2 n_3, n_3 n_1\}$ for 3D), we have
\begin{equation}\label{eq:RelationshipAmongMethodsBigQ}
    \mathcal{M}^h_Q=\mathcal{V}^h\subset\mathcal{G}^h_Q \subset \mathcal{H}^h \subset \mathcal{N}^h, 
\end{equation}
as illustrated in Fig. \ref{fig:Relationship}.

\begin{figure}[h]
	\centering
	\includegraphics[width=0.8\textwidth]{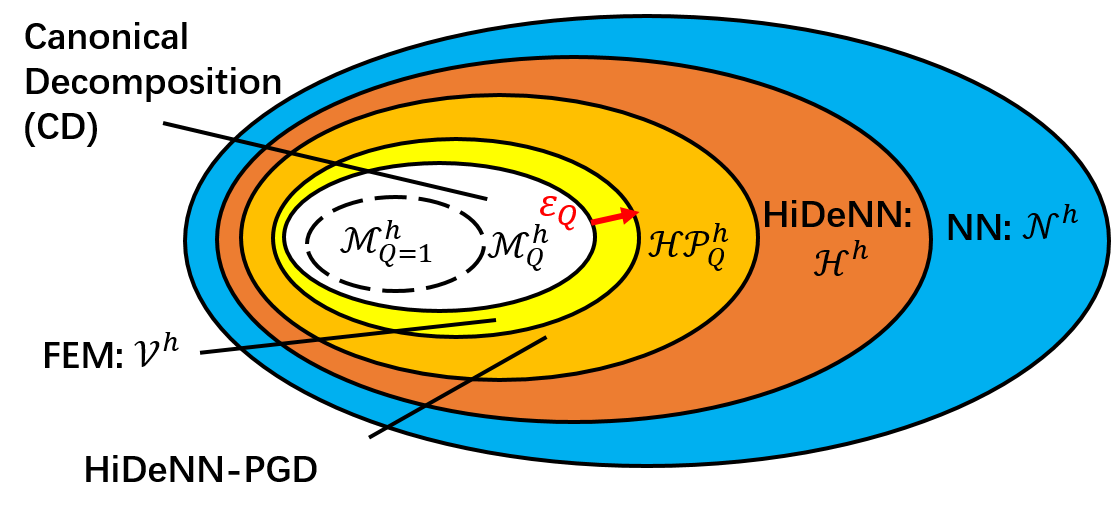}
	\caption{Illustration for relationship among interpolation function sets of CD, FEM, HiDeNN-PGD, HiDeNN and NN. Especially, when $Q$ tends to infinity, $\mathcal{M}_Q^h$ approaches to $\mathcal{V}^h$. When $Q$ is big enough, FEM interpolation set $\mathcal{V}^h$ is the subset of HiDeNN-PGD one $\mathcal{G}^h_Q$.}
	\label{fig:Relationship}
\end{figure}

The above conclusions (\ref{eq:RelationshipAmongMethods})-(\ref{eq:RelationshipAmongMethodsBigQ}) are based on the following observations:
\begin{description}
	\item[$\bullet$] According to (\ref{eq:MSExpand}), the CD interpolation functions can be regarded as finite element interpolation functions, which belong to $\mathcal{V}^h$, so $\mathcal{M}_Q^h \subset \mathcal{V}^h$. Especially, when $Q$ is big enough ($Q\geq\min\{n_1,n_2\}$ for 2D and $Q\geq\min\{n_1 n_2,n_2 n_3, n_3 n_1\}$ for 3D), $\mathcal{M}_Q^h$ approaches to $\mathcal{V}^h$, i.e., $\mathcal{M}^h_Q=\mathcal{V}^h$. Detailed results will be shown in Section 4. Proofs can be found in Appendix \ref{Appdix:2DConverge}.
	\item[$\bullet$] In HiDeNN, an optimization of the nodal positions is performed. Thus FEM may be regarded as a specific case in the HiDeNN with nodal coordinates fixed .
	\item[$\bullet$] HiDeNN is a class of structured NN with weights and biases as functions of nodal values and nodal positions.
	\item[$\bullet$] HiDeNN-PGD requires the mesh regular, while HiDeNN optimizes nodal positions freely, so $\mathcal{G}^h_Q \subset \mathcal{H}^h, \forall Q \in \mathbb{N}$. On the other hand, HiDeNN-PGD has more DoFs than CD under the same number of modes, so $\mathcal{M}^h_Q \subset \mathcal{G}^h_Q, \forall Q \in \mathbb{N}$. When $Q$ is small, $\mathcal{M}^h_Q$ is the subset of the intersection of $\mathcal{G}^h_Q$ and $\mathcal{V}^h$.
\end{description}

We can also summarize the DoFs for different methods in Table \ref{table:dofsmethod}. It is shown that HiDeNN-PGD and PGD have only a linear growth in terms of DoFs, whereas the DoFs of FEM and HiDeNN may grow in a polynomial manner.
\begin{table}[!htb]
\caption{Comparison of DoFs for different methods on a 3D mesh}
\centering
\begin{tabular}{|c | c | c | c | c |}
\hline
& FEM& PGD/CD& HiDeNN-PGD&HiDeNN \\ \hline
DoFs & $n_1 \times n_2 \times n_3$& $(n_1 + n_2 +  n_3) \times Q$ & $(n_1 + n_2 +  n_3) \times Q +n_1 + n_2 +  n_3$  & $n_1 \times n_2 \times n_3 +n_1 \times n_2 \times n_3$\\ \hline
\end{tabular}
\label{table:dofsmethod}
\end{table}

\section{Error analysis of FEM, CD, HiDeNN, NN-based solutions and HiDeNN-PGD for PDEs}

We consider a partial differential equation with homogeneous boundary conditions
\begin{equation}
\left\{\begin{array}{l}
    \mathcal{L}\bm{u}(\bm{x})+\bm{b}(\bm{x})=0 \text{ in } \Omega \subset \mathbb{R}^d, \\
    \bm{u}|_{\partial \Omega}=\bm{0},
\end{array}\right.
\label{eq:GeneralPDE}
\end{equation}
where $\bm{u}\in \mathbb{R}^m$ denotes an $m$-dimensional vector-valued function in a certain Hilbert space $H(\mathbb{R}^d, \mathbb{R}^m)$, $\bm{b}$ the source term, $\bm{x} \in \mathbb{R}^d$ the $d$-dimensional space coordinates, and $\mathcal{L}$ a second-order differential operator.

We assume that an energy potential $\Pi(u)$ exists, formulated in the following form
\begin{equation}
    \Pi(\bm{u})=\dfrac{1}{2}a(\bm{u},\bm{u})-(\bm{b},\bm{u}),
\end{equation}
where $a(\cdot , \cdot)$ is the symmetric and bilinear form corresponding to the second-order differential operator $\mathcal{L}$, and $(\bm{f},\bm{g})=\int_{\Omega} \bm{f} \cdot \bm{g} \mathrm{d}\bm{x}$ denotes the inner product. For example, let $a(\bm{f},\bm{g})=\int_{\Omega} \nabla \bm{f} \cdot \nabla \bm{g} \mathrm{d}\bm{x}$ for Poisson equation. 
The minimization of $\Pi(u)$ gives the solution to \eqref{eq:GeneralPDE}. This leads to a weak form of  \eqref{eq:GeneralPDE}, which reads
\begin{equation}
    \bm{u}=\argmin_{\bm{u}^* \in H(\mathbb{R}^m, \mathbb{R}^n)} \Pi[\bm{u^*}].
\end{equation}

Such a weak form is commonly adopted in interpolation theory based numerical approaches, such as the methods shown in Section 2. Denoted by $\mathcal{S}^h \subset H(\mathbb{R}^d, \mathbb{R}^m)$ the discretized approximation solution set with a characteristic mesh size $h$, the approximate solution based on this given interpolation function set is then
\begin{equation}
    \bm{u}^h=\argmin_{\bm{u}^{h*} \in \mathcal{S}^h} \Pi[\bm{u^{h*} }].
\end{equation}
In the following, we shall take $\mathcal{S}^h$ to be $\mathcal{M}_{0,Q}^h, \mathcal{V}_0^h, \mathcal{H}_0^h, \mathcal{N}_0^h, \mathcal{G}_{0,Q}^h$, which are the subsets of $\mathcal{M}_Q^h, \mathcal{V}^h, \mathcal{H}^h, \mathcal{N}^h, \mathcal{G}_{Q}^h$ under homogeneous boundary conditions,  respectively. 

\subsection{Error analysis}

We assert the following relations among error bounds for FEM, NN, HiDeNN and CD
\begin{equation}
    \left\| u^{CD}-u^{exact} \right\|_{E} \geq \left\| u^{FEM}-u^{exact} \right\|_{E} \geq \left\| u^{HiDeNN}-u^{exact} \right\|_{E} \geq \left\| u^{NN}-u^{exact} \right\|_{E},
    \label{eq:ErrorAnalysis}
\end{equation}
Here, $\left\| \cdot \right\|_{E}=\sqrt{a(\cdot,\cdot)}$ is called as the energy norm, and $u^{exact}$ is the real solution of the problem, i.e.,
\begin{equation}
    u^{exact}=\arg \min_{u\in H^1_0} \Pi[u].
\end{equation}

Especially when $Q$ is big enough ($Q\geq\min\{n_1,n_2\}$ for 2D and $Q\geq\min\{n_1 n_2,n_2 n_3, n_3 n_1\}$ for 3D), the error bounds for five methods become
\begin{eqnarray}
    &&\left\| u^{CD}-u^{exact} \right\|_{E} \geq \left\| u^{FEM}-u^{exact} \right\|_{E} \geq \left\| u^{HiDeNN-PGD}-u^{exact} \right\|_{E} \\ \nonumber
    &&\geq \left\| u^{HiDeNN}-u^{exact} \right\|_{E} \geq \left\| u^{NN}-u^{exact} \right\|_{E},
    \label{eq:ErrorAnalysisUpdate}
\end{eqnarray}

The relationship (\ref{eq:RelationshipAmongMethods}) is inherited by $\mathcal{M}_{0,Q}^h, \mathcal{V}_0^h, \mathcal{H}_0^h, \mathcal{N}_0^h$, i.e.,
\begin{equation}
    \mathcal{M}_{0,Q}^h \subset \mathcal{V}_0^h \subset \mathcal{H}_0^h \subset \mathcal{N}_0^h.
    \label{eq:RelationshipAmongMethodsBC}
\end{equation}

These four methods are all based on the minimal energy principle,
\begin{equation}
    \bm{u}=\arg \min_{\bm{u}^h \in \mathcal{S}^h} \Pi[\bm{u}^h],
\end{equation}
where $\mathcal{S}^h$ is selected as $\mathcal{M}_{0,Q}^h, \mathcal{V}_0^h, \mathcal{H}_0^h, \mathcal{N}_0^h$ for CD, FEM, HiDeNN, and NN, respectively. Thus due to the relationship (\ref{eq:RelationshipAmongMethodsBC}) among them, we have
\begin{equation}
    \Pi[u^{CD}] \geq \Pi[u^{FEM}] \geq \Pi[u^{HiDeNN}] \geq \Pi[u^{NN}].
\end{equation}
This leads to the error bounds (\ref{eq:ErrorAnalysis}). 
In the same manner, the relationship (\ref{eq:RelationshipAmongMethodsBigQ}) leads to the error bounds (\ref{eq:ErrorAnalysisUpdate}). 
(\ref{eq:ErrorAnalysisUpdate}) shows that HiDeNN-PGD might reach better accuracy than FEM with regular mesh at increasing number of modes.

We remark that the above theoretical analysis does not account for numerical aspects, such as the difficulties to find out the global minimizer of the energy potential.  

\subsection{Proof of the mesh in-dependency of mode reduction error in the PGD method}

The PGD based model reduction induces two kinds of errors: mesh discretization error and mode reduction error. 

\begin{theorem}
Let $u^{exact}, u^{PGD}$,$u^{FEM}$ be the exact solution, the numerical solution of PGD and FEM, respectively. PGD and FEM take the same regular mesh, and FEM takes the shape functions as the product of 1D shape functions of PGD in each dimension. Then the following error decomposition holds:
\begin{equation}\label{eq:ErrorDecomposition}
    \| u^{PGD}-u^{exact} \|_E^2 = \| u^{FEM}-u^{exact} \|_E^2 + \| u^{PGD}-u^{FEM} \|_E^2.
\end{equation}
\end{theorem}

\begin{proof}
We calculate
\begin{eqnarray} \label{eq:PGDDecompExpand}
    && \| u^{PGD}-u^{exact} \|_E^2 \\
    &=&\int_\Omega (\nabla u^{PGD}- \nabla u^{FEM}+\nabla u^{FEM}-\nabla u^{exact})^2 \mathrm{d} \bm{x} \\ \nonumber
    &=&\| u^{FEM}-u^{exact} \|_E^2 + \| u^{PGD}-u^{FEM} \|_E^2+\int_\Omega 2(\nabla u^{PGD}- \nabla u^{FEM})\cdot(\nabla  u^{FEM}-\nabla u^{exact}) \mathrm{d} \bm{x}.
\end{eqnarray}
By the Gauss theorem, we obtain
\begin{eqnarray} \label{eq:PGDDecompCross}
    &&\int_\Omega (\nabla u^{PGD}- \nabla u^{FEM})\cdot(\nabla u^{FEM}-\nabla u^{exact}) \mathrm{d} \bm{x} \\ \nonumber
    &=&\int_\Omega (\nabla u^{PGD}- \nabla u^{FEM})\cdot \nabla u^{FEM}+(u^{PGD}-u^{FEM})\nabla^2 u^{exact} \mathrm{d} \bm{x} \\ \nonumber
    &=&\int_\Omega (\nabla u^{PGD}- \nabla u^{FEM})\cdot \nabla u^{FEM}-(u^{PGD}-u^{FEM}) b(\bm{x}) \mathrm{d} \bm{x}.
\end{eqnarray}
Since PGD and FEM share the same mesh and shape functions, $v= u^{PGD}- u^{FEM}$ belongs to the test function space of FEM. By the weak form of the FEM problem, we have
\begin{equation}
    \int_\Omega \nabla v \cdot \nabla u^{FEM}-v(\bm{x})b(\bm{x}) \mathrm{d} \bm{x} = 0.
\end{equation}
That is to say, (\ref{eq:PGDDecompCross}) vanishes and hence (\ref{eq:ErrorDecomposition}) holds.
$\hfill\square$.
\end{proof}

This theorem asserts that the PGD error is a direct sum of the FEM error and the mode reduction error (the difference between FEM and PGD).

%% file: Part2_HiDeNNPGD.tex
\section{The formulation of HiDeNN-PGD: the 2D Poisson problem as illustration}
In subsection \ref{ssec:HiDeNN-PGD}, we defined HiDeNN-PGD in terms of the approximation function space. Here, we give the detailed formulation of the method.

For the sake of simplicity and without loss of generality, we consider 2D Poisson problem,
\begin{equation}
\left\{\begin{array}{l}
    \Delta u(x,y)+b(x,y)=0 \text{ in } \Omega_{(x,y)} \subset \mathbb{R}^2, \\
    u|_{\partial \Omega}=0,
\end{array}\right.
\label{eq:PoissonEqHP}
\end{equation}
(\ref{eq:PoissonEqHP}) is solved in the regular domain $\Omega_{(x,y)}=[a,b]\times[c,d]$ with homogeneous boundary conditions. Note that if inhomogeneous boundary conditions are under consideration, we can separate the solution into two parts,
\begin{equation}
    u=u^0+\tilde{u},
\end{equation}
where $u^0$ is an arbitrary function satisfying boundary conditions, and $\tilde{u}$ is the solution to the new equations with homogeneous boundary condition.
The variational formula of (\ref{eq:PoissonEqHP}) is 
\begin{equation}
    \Pi(u)=\dfrac{1}{2}\int_{\Omega_{(x,y)}} \left|\nabla u\right|^2 \mathrm{d}x\mathrm{d}y - \int_{\Omega_{(x,y)}} u(x,y)b(x,y) \mathrm{d}x\mathrm{d}y
    \label{eq:PoissonVar}
\end{equation}

Substituting HiDeNN-PGD interpolation function into (\ref{eq:PoissonVar}), we obtain
\begin{eqnarray}
    \label{VariationalHP}
    \Pi(u^h)&=&\dfrac{1}{2}\sum_{p=1}^Q\sum_{q=1}^Q \left(\int_{x_1}^{x_{n_1}} \dfrac{d}{dx}X^{(p)}(x)\dfrac{d}{dx}X^{(q)}(x) \mathrm{d}x\right) \left(\int_{y_1}^{y_{n_2}} Y^{(p)}(y)Y^{(q)}(y) \mathrm{d}y\right) \\ \nonumber
    &+&\dfrac{1}{2}\sum_{p=1}^Q\sum_{q=1}^Q \left(\int_{x_1}^{x_{n_1}} X^{(p)}(x)X^{(q)}(x) \mathrm{d}x\right) \left(\int_{y_1}^{y_{n_2}} \dfrac{d}{dy}Y^{(p)}(y) \dfrac{d}{dy}Y^{(q)}(y) \mathrm{d}y\right) \\ \nonumber
    &-&\sum_{q=1}^Q\int_{\Omega_{(x,y)}} X^{(q)}(x)Y^{(q)}(y) b(x,y)\mathrm{d}x\mathrm{d}y,
\end{eqnarray}
with the discrete mesh $[x_1=a, x_2, \cdots, x_{n_1}=b]\times[y_1=c,y_2,\cdots,y_{n_2}=d]$.
Notice that there exist cross terms in Eq. (\ref{VariationalHP}). For convenience and considering the difficulties to do exact integration between different discrete meshes, all the modes share the same mesh $[x_1,x_2,\cdots,x_{n_1}]\times[y_1,y_2,\cdots,y_{n_2}]$ and the same shape functions. We use Gauss quadrature for the source term.

Once the interpolation function set (\ref{eq:HiDeNNPGD_FunSet}) is obtained, minimal variational principle gives the approximated solution. The process in HiDeNN-PGD is formulated as
\begin{eqnarray}
    \text{find} && \beta_I^{(1)}, \gamma_J^{(1)},\cdots,\beta_I^{(Q)}, \gamma_J^{(Q)}, x_I, y_J,I=1,\cdots,n_1, J=1,\cdots,n_2\\ \nonumber
    \text{min} && \dfrac{1}{2}\int_{\Omega_{(x,y)}} \left|\nabla u\right|^2 \mathrm{d}x\mathrm{d}y - \int_{\Omega_{(x,y)}} u(x,y)b(x,y) \mathrm{d}x\mathrm{d}y \\ \nonumber
    && u^h=\sum_{q=1}^Q \left( \sum_{I=1}^{n_1} \mathcal{N}_I(x;\bm{x}_I^*,\bm{\mathcal{A}})\beta_I^{(q)} \right) \left( \sum_{J=1}^{n_2} \mathcal{N}_J(y;\bm{y}_J^*,\bm{\mathcal{A}})\gamma_J^{(q)} \right) \\ \nonumber
    \text{and} && \sum^{n_1}_{I=1}\mathcal{N}(\bm{x}^{*}_{I},x,\mathcal{A}) = 1, \sum^{n_2}_{J=1}\mathcal{N}(\bm{y}^{*}_{J},y,\mathcal{A}) = 1.
\end{eqnarray}
The gradient descent method is applied to iteratively minimize $\Pi(\mathscr{u}^h)$ and solve for all parameters together. In the following numerical examples, we choose Adam algorithm \cite{kingma2014adam}, i.e.,
\begin{framed}
	\begin{enumerate}
		\item Initialization: Set number of modes $Q$, initial nodal positions $x_I,y_J,I=1,2,\cdots,n_1,J=1,2,\cdots,n_2$, initial coefficients $\beta^{(q)}_I,\gamma^{(q)}_J, q=1,2,\cdots,Q$ and maximal iteration step $M$
		\item Algorithm: 
		
		While $k\leq M$ do
		\begin{enumerate}
		    \item Compute gradient $\dfrac{\partial}{\partial x_I}, \dfrac{\partial}{\partial y_J}, \dfrac{\partial}{\partial \beta_I^{(q)}},
		    \dfrac{\partial}{\partial \gamma_J^{(q)}}, I=1,\cdots,n_1, J=1,\cdots,n_2, q=1,\cdots,Q$;
		    \item Update $x_I,y_J,\beta_I^{(q)},\gamma_J^{(q)},I=1,\cdots,n_1, J=1,\cdots,n_2, q=1,\cdots,Q$ by using Adam algorithm;
		\end{enumerate}
		End while.
	\end{enumerate}
\end{framed}

%% file: NumericalExamples.tex
\section{Numerical examples}
In this section, we study the performance of HiDeNN-PGD with comparison to FEM, HiDeNN and PGD methods. This study mainly focuses on the accuracy comparison and the convergence behavior, as the computational cost can be strongly affected by the under-optimized implementation. The computational efficiency of the proposed HiDeNN-PGD method will be investigated in our future work on the basis of GPU-computing.  


\subsection{2D case }
The HiDeNN-PGD is applied to the Poisson problem with a concentrated load. To demonstrate the capability of the method, the domain analyzed by HiDeNN-PGD is initialized with a uniform mesh of 40 by 40 elements. As shown in \figurename~\ref{fig:HiDPGDtoFEM}, the final solution agrees with the reference FE solution. This reference solution is obtained with a very fine mesh containing $4,000\times 4,000$ elements. 

\tablename~\ref{table:error} illustrates the evolution of accuracy with an increasing number of modes. For comparison purposes, we also applied the PGD, CD, FEM and HiDeNN  for the same problem on the uniform mesh of 40 by 40. The errors of different methods are computed based on the energy norm with comparison to the reference solution. As expected,  HiDeNN is the most accurate one but leads to a significantly larger number of degrees of freedom (DoFs). The proposed HiDeNN-PGD method can have the same level of accuracy and only requires a small number of modes. Compared to PGD and CD, the HiDeNN-PGD is much more accurate at a limited number of modes. Taking a closer look at PGD and CD, these two methods converge overall to the FEM on the coarse mesh. HiDeNN-PGD and HiDeNN can overcome this limitation imposed by the mesh size with the adaptivity. This observation is consistent with our theoretical analysis.  Moreover, it should be noticed that, unlike HiDeNN, the HiDeNN-PGD only increases slightly the DoFs when compared with PGD and CD. This attributes to the separation of variables. Indeed, the mesh adaptation is performed only in the two separated axes, as shown in \figurename~\ref{fig:HiDPGDtoFEM}(d). For comparison purposes, the final optimized mesh of HiDeNN is illustrated in \figurename~\ref{fig:MeshForHPandHiDeNN}. It is shown that the HiDeNN enables a full adaptivity for the entire mesh, which leads to a significantly different and more accurate final result. Nevertheless, the HiDeNN-PGD shows attractive advantages in terms of DoFs.

\figurename~\ref{fig:mode} illustrates the first four modes of HiDeNN-PGD, PGD, and CD. The PGD modes remain similar to CD in this example. However, the modes of HiDeNN-PGD seem to be more concentrated in the region of interest. This difference mainly comes from the mesh adaptivity.

To further confirm the performance of HiDeNN-PGD, we have compared the accuracy of different methods on different meshes. In \tablename~\ref{table:errormesh}, the PGD, CD and HiDeNN-PGD results are obtained from the final converged mode. It is shown that the HiDeNN-PGD always gives more accurate results with fewer degrees of freedom. This confirms the previous observation.

\begin{figure}[htbp]
\centering
\subfigure[Reference FE solution obtained with a extremely fine mesh ]{\includegraphics[scale=0.9]{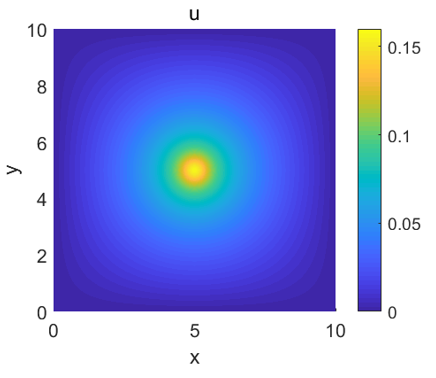}}\quad \quad
\subfigure[HiDeNN-PGD solution]{\includegraphics[scale=0.9]{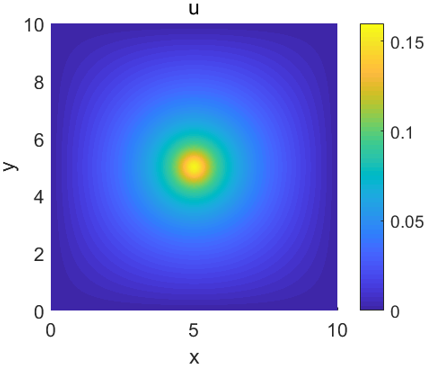}}
\subfigure[Initial uniform mesh for HiDeNN-PGD ]{\includegraphics[scale=1]{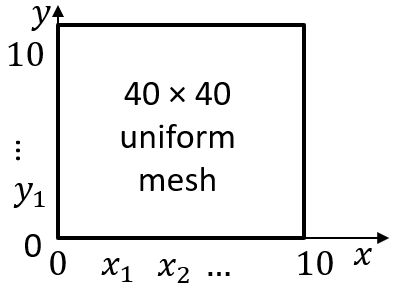}}
\subfigure[Final optimized mesh for HiDeNN-PGD]{\includegraphics[scale=0.9]{figures/HiDeNNPGDmesh.PNG}}
\caption{FE solution versus HiDeNN-PGD solution.}
\label{fig:HiDPGDtoFEM}
\end{figure}

\begin{figure}[htbp]
\centering
\subfigure[Inital mesh for HiDeNN ]{\includegraphics[scale=0.5]{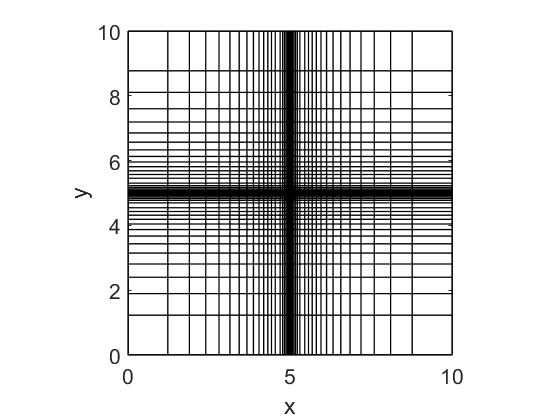}}
\subfigure[Final optimized mesh for HiDeNN]{\includegraphics[scale=0.5]{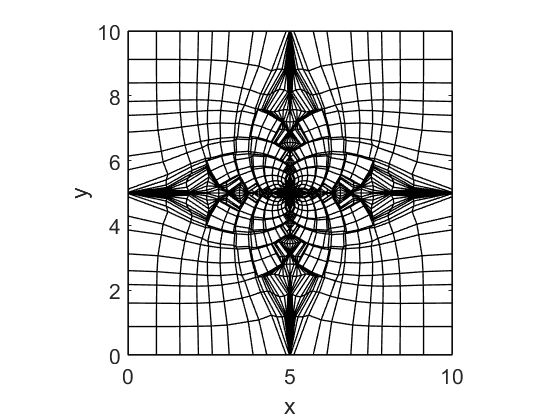}}
\caption{Mesh optimization in HiDeNN.}
\label{fig:MeshForHPandHiDeNN}
\end{figure}

\begin{figure}[htbp]
\centering
\subfigure[First four modes in HiDeNN-PGD ]{\includegraphics[scale=0.7]{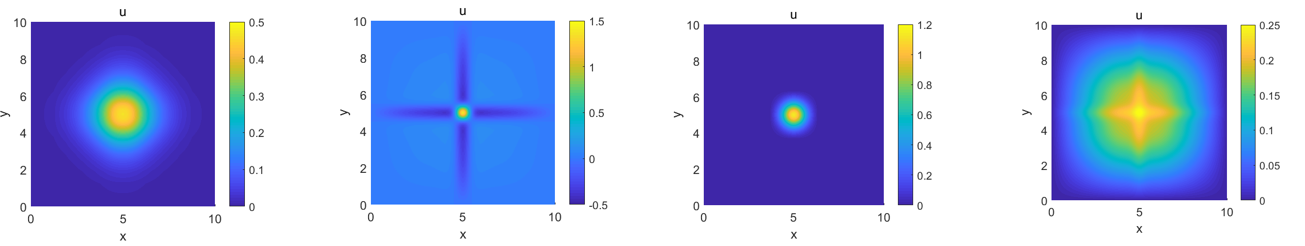}}
\subfigure[First four modes in PGD]{\includegraphics[scale=0.7]{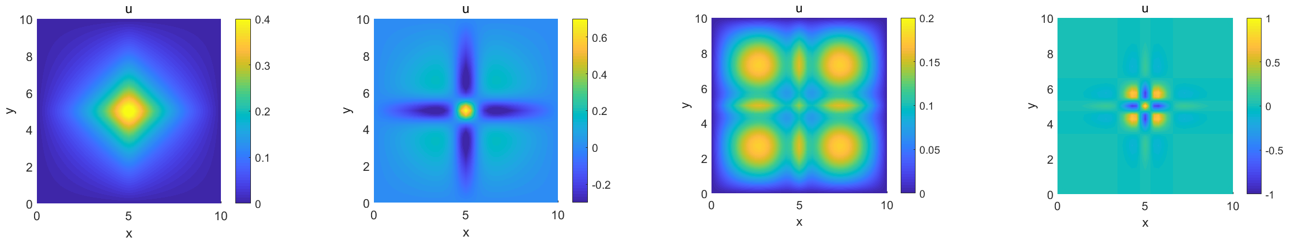}}
\subfigure[First four modes in CD ]{\includegraphics[scale=0.7]{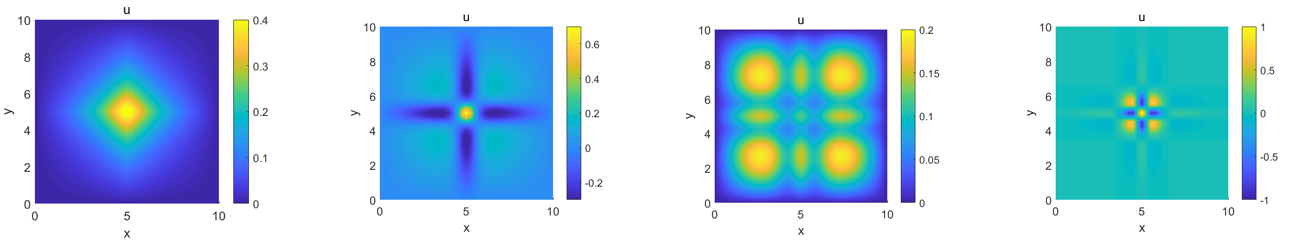}}
\caption{Mode comparison for HiDeNN-PGD, PGD and CD}
\label{fig:mode}
\end{figure}

\begin{table}[!htb]
\caption{Accuracy comparison for different methods on the coarse 40 by 40 mesh}
\centering
\begin{tabular}{|c | c c | c c | c c | c c | c c|}
\hline
 & \multicolumn{2}{c}{PGD}& \multicolumn{2}{|c}{CD} &\multicolumn{2}{|c}{FEM}&\multicolumn{2}{|c}{HiDeNN-PGD}&\multicolumn{2}{|c|}{HiDeNN} \\ \hline
Mode number & DoFs& Err & DoFs& Err & DoFs& Err& DoFs& Err & DoFs& Err \\ \hline
1 & 78& $38.167\%$ & 78& $38.167\%$ & \textbf{1521}& $\bm{11.659\%}$& 156& $37.357\%$ & \textbf{4719}& $\bm{2.102\%}$ \\ \hline
2 & 156& $16.500\%$ & 156& $14.422\%$ & &- & \textbf{234}& $\bm{9.293\%}$ & & - \\ \hline
3 & 234& $13.188\%$ & 234& $11.789\%$ & &- & 312& $3.674\%$ & & - \\ \hline
4 & 312& $11.811\%$ &312& $11.664\%$ & &- & 390& $3.662\%$ & & - \\ \hline
5 & 390& $11.685\%$ &\textbf{390}& $\bm{11.659\%}$ & &- &468 & $3.661\%$ & & - \\ \hline
6 & 468& $11.666\%$ &468& $11.659\%$ & &- &546& $3.661\%$ & & - \\ \hline
8 & \textbf{624}& $\bm{11.659\%}$ & & - & &- && - & & - \\ \hline
20 & 1560& $11.659\%$ & & - & &- && - & & - \\ \hline
\end{tabular}
\label{table:error}
\end{table}

\begin{table}[!htb]
\caption{Accuracy comparison for different methods with different meshes }
\centering
\begin{tabular}{|c | c c | c c | c c | c c | c c|}
\hline
 & \multicolumn{2}{c}{PGD}& \multicolumn{2}{|c}{CD} &\multicolumn{2}{|c}{FEM}&\multicolumn{2}{|c}{HiDeNN-PGD}&\multicolumn{2}{|c|}{HiDeNN} \\ \hline
Mesh & DoFs& Err & DoFs& Err & DoFs& Err& DoFs& Err & DoFs& Err \\ \hline
$40\times40$ & 624(8)
& $11.659\%$ & 390(5)
& $11.659%
\%$ & \textbf{1,521}
& $\bm{11.659\%}$& \textbf{468}(5)
& $\bm{3.661\%}$ & \textbf{4719}& $\bm{2.102\%}$ \\ \hline
$80\times80$ & 1,422(9)
& $5.887\%$ & 790(5)
& $5.887\%$ &6,241
 &$5.887\%$& 1,106(6)
& $1.851\%$ & 19,039
 & $1.406\%$\\ \hline
$160\times160$ & 2,862(9)
& $2.948\%$ & 1,908(6)
& $2.948\%$ &25,281
 &$2.948\%$ & 2,226(6)
& $1.174\%$ &76,479
 & $1.081\%$  \\ \hline
$320\times320$ & 7,018(11)
& $1.469\%$ &5,104(8)
& $1.469\%$ &101,761
 &$1.469\%$ & 3,828(5)
& $0.896\%$ &306,559
 & $0.889\%$\\ \hline
 $640\times640$ & 11,502(9)
& $0.724\%$ &7,668(6)
& $0.724\%$ &408,321
 &$0.724\%$ & 10,224(7)
& $0.606\%$ &1,227,519
 & $0.597\%$ \\ \hline
\end{tabular}
\label{table:errormesh}
Note that the number in the parentheses indicates the number of modes. 
\end{table}

\subsubsection{Convergence studies}
In the HiDeNN-PGD method, the mode number has to be prescribed. In general, this is unknown for a given problem and can vary significantly from one to another. Thus, we want to study the convergence property of this method and the PGD to get a general idea about how to choose the mode number. 

As a first attempt, we restrict ourselves to a one-mode solution problem. This eliminates the effect of the number of modes, and allows us to study the convergence rate of the methods with respect to mesh refinement. This kind of error is usually known as discretization error in FEM. To do so, the body force term is manufactured so that the final solution is analytically known in a separated form. As shown in \figurename~\ref{fig:convdof}, the PGD and HiDeNN-PGD results converge with respect to the element size at a rate similar to FEM. However, if we consider the DoFs, this error converges much faster for PGD and HiDeNN-PGD. This confirms that doing separation of variables for PGD and HiDeNN-PGD does not degrade the convergence rate in terms of mesh refinement.

The PGD based model reduction induces two kinds of errors: mesh discretization error and mode reduction error. In particular, we have theoretically shown that the latter one should be independent of the mesh. In order to show this numerically, we use a manufactured load to compute the multi-modes PGD solution in different meshes. The results are shown in \figurename~\ref{fig:convPGD2D}. As expected, the convergence rates on the number of modes remain similar regardless of the mesh size. It seems that the log error is linearly proportional to the mode number.  The decreasing slop remained unchanged from the very coarse mesh to the fine one. This implies the coarse mesh has the same mode reduction error as a fine mesh. Hence, it can be used to choose the mode number.

From the above observation, we may consider using a coarse mesh PGD, which is very cheap, to study the mode reduction error and choose an appropriate mode number for HiDeNN-PGD. Since the the HiDeNN-PGD is always more accurate than the usual PGD, the selected prescribed number should be large enough.

\begin{figure}[htbp]
\centering
\subfigure[Convergence with the element size ]{\includegraphics[width=3in]{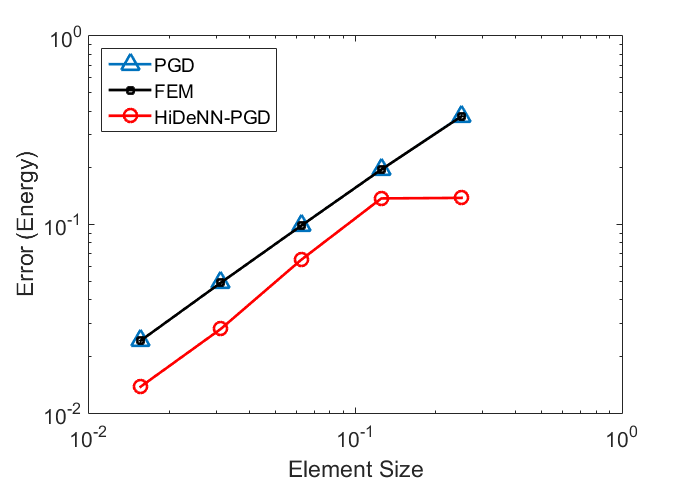}}
\subfigure[Convergence with DoFs]{\includegraphics[width=3in]{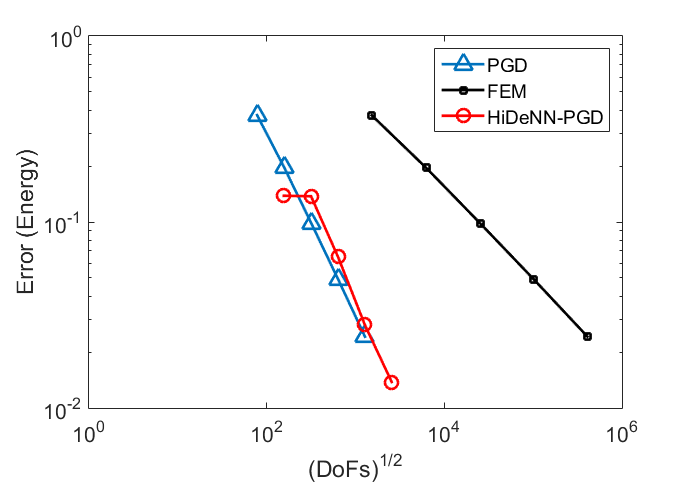}}
\caption{Convergence rate of FEM, PGE, HiDeNN-PGD with respect to mesh refinement}
\label{fig:convdof}
\end{figure}

\begin{figure}[htbp]
\centering
{\includegraphics[scale=0.9]{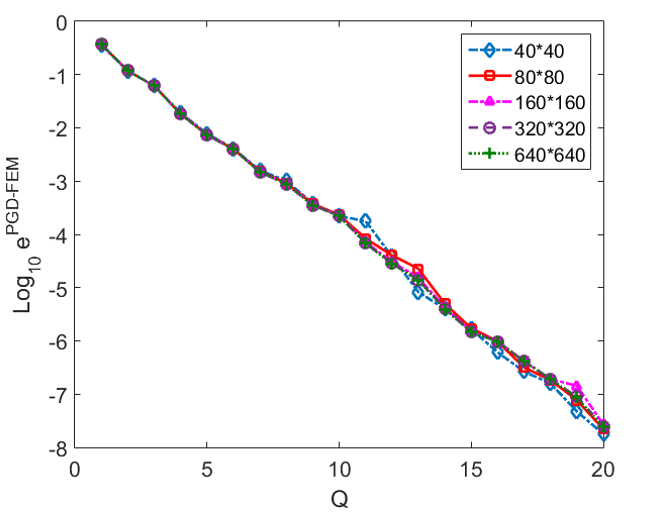}}
\caption{Convergence of PGD with respect to the increasing number of modes for different meshes}
\label{fig:convPGD2D}
\end{figure}

\subsection{3D case }
The proposed HiDeNN-PGD method has been tested in three dimensional cases. Similar to the previous two-dimensional example, \tablename~\ref{table:3Derror} reports the evolution of error at an increasing mode number for different methods on a coarse mesh. Again, the HiDeNN-PGD outperforms the other methods in terms of accuracy and DoFs. The same conclusion can be drawn on finer meshes, as reported in \tablename~\ref{table:3Derrormesh}.

\begin{table}[!htb]
\caption{Accuracy comparison for different methods on the coarse $40\times40\times40$ mesh}
\centering
\begin{tabular}{|c | c c | c c | c c | c c | c c |}
\hline
 & \multicolumn{2}{c}{PGD}& \multicolumn{2}{|c}{CD} &\multicolumn{2}{|c}{FEM}&\multicolumn{2}{|c|}{HiDeNN-PGD}&\multicolumn{2}{|c|}{HiDeNN} \\ \hline
Mode number & DoFs& Err & DoFs& Err & DoFs& Err& DoFs& Err& DoFs& Err  \\ \hline
1 & 117& $72.474\%$ & 117& $72.474\%$ & \textbf{59,319}& $\bm{27.870\%}$ & 234& $72.349\%$ & \textbf{255,996}
& $\bm{10.780\%}$  \\ \hline
2 & 234& $29.883\%$ & 234& $27.927\%$ & &- & \textbf{351}& $\bm{10.797\%}$ & &-\\ \hline
3 & 351& $28.026\%$ & 351& $27.923\%$ & &- & 468& $10.797\%$ & &- \\ \hline
4 & 468& $27.931\%$ &468& $27.920\%$ & &- & 585& $10.797\%$ & &- \\ \hline
5 & 585& $27.895\%$ &585& $27.872\%$ & &- & 702 & $10.796\%$ & &- \\ \hline
6 & 702& $27.881\%$ &702& $27.871\%$ & &- & - & -& &-\\ \hline
7 & 819& $27.877\%$ &819& $27.871\%$ & &- & - & -& &-\\ \hline
8 & 936& $27.874\%$ &\textbf{936}& $\bm{27.870\%}$ & &- & - & -& &-\\ \hline
17 & \textbf{1,989}& $\bm{27.870\%}$ & & - & &- &  & -& &-\\ \hline
20 & 2,340& $27.870\%$ & & - & &- & - & -& &-\\ \hline
\end{tabular}
\label{table:3Derror}
\end{table}

\begin{table}[!htb]
\caption{Accuracy comparison for different methods with different meshes for 3D problem }
\centering
\begin{tabular}{|c | c c | c c | c c | c c | c c | }
\hline
 & \multicolumn{2}{c}{PGD}& \multicolumn{2}{|c}{CD} &\multicolumn{2}{|c}{FEM}&\multicolumn{2}{|c|}{HiDeNN-PGD}&\multicolumn{2}{|c|}{HiDeNN} \\ \hline
Mesh & DoFs& Err & DoFs& Err & DoFs& Err& DoFs& Err & DoFs& Err \\ \hline
$40\times40\times40$ & 1,989(17)
& $27.870\%$ & 936(8)
& $27.870%
\%$ & \textbf{59,319}
& $\bm{27.870\%}$& \textbf{702}(5)
& $\bm{10.796\%}$ & \textbf{255,996}
& $\bm{10.780\%}$ \\ \hline
$80\times80\times80$ & 2,607(11)
& $14.416\%$ & 1185(5)
& $14.416\%$ &493,039
 &$14.416\%$& 1185(4)
& $6.771\%$ & 2,047,996
& \\ \hline
$160\times160\times160$ & 4,770(10)
& $7.247\%$ & 2,385(5)
& $7.247\%$ &4,019,679
 &$7.247\%$ & 3,339(6)
& $4.036\%$ & 16,383,996 &  \\ \hline
$320\times320\times320$ & 9,570(10)
& $3.628\%$ &5,742(6)
& $3.628\%$ &32,461,759
 &$3.628\%$ & 5,742(5)
& $1.831\%$ & 131,071,996
& \\ \hline
\end{tabular}
\label{table:3Derrormesh}
\end{table}


%% file: Appendix.tex
\section{Convergence for general canonical decomposition method at increasing number of modes}

In this section, we discuss the convergence of the canonical decomposition method at increasing $Q$ (number of modes). In Sect. 2, we have shown that $\mathcal{M}_Q^h \subset \mathcal{V}^h$, provided that their interpolations are based on the same basis functions. Here, we make some further discussions.

For 2D case, we compare 
\begin{equation}
    \mathcal{M}_Q^h=\left\{u^h \bigg |u^h=\sum_{q=1}^Q \left( \sum_{I=1}^{n_1} N_I(x)\beta_I^{(q)} \right) \left( \sum_{J=1}^{n_2} N_J(y)\gamma_J^{(q)} \right), \beta_I^{(q)}, \gamma_J^{(q)} \in \mathbb{R}, I=1,\cdots,n_1, J=1,\cdots,n_2\right\}
    \label{eq:2DMS}
\end{equation}
with
\begin{equation}
    \mathcal{V}^h=\left\{ u^h \bigg |u^h=\sum_{I=1}^{n_1} \sum_{J=1}^{n_2} N_I(x)N_J(y)u_{(I,J)}, u_{(I,J)}\in\mathbb{R}, I=1,\cdots,n_1, J=1,\cdots,n_2 \right\}.
    \label{eq:2DFEMVh}
\end{equation}
As the same basis functions are used, it follows that:
\begin{description}
	\item[1.] $\forall Q\in\mathbb{N}, \mathcal{M}_Q^h \subset \mathcal{V}^h$;
	\item[2.] If $Q_1 \le Q_2$, $Q_1,Q_2 \in \mathbb{N}$, $\mathcal{M}_{Q_1}^h \subset \mathcal{M}_{Q_2}^h$;
	\item[3.] $\forall Q\geq \min\{n_1,n_2\}, Q\in\mathbb{N},  \mathcal{M}_Q^h = \mathcal{V}^h$.
\end{description}
The first two statements are straightforward. The last property is proved in Appendix A.

We then conclude the following relationship:
\begin{equation}
    \mathcal{M}^h_{Q=1} \subset \mathcal{M}^h_{Q=2} \subset \cdots \subset \mathcal{M}^h_{Q=\min\{n_1,n_2\}} = \mathcal{M}^h_{Q=\min\{n_1,n_2\}+1} = \cdots = \mathcal{V}^h.
\end{equation}
In other words, $\mathcal{M}^h_{Q}$ is always a subset of $\mathcal{V}^h$, and it becomes the same as $\mathcal{V}^h$ when $Q$ increases. Consequently, when enough number of modes are taken, the canonical decomposition result reaches the same accuracy as the FEM solution. We remark that $\min\{n_1,n_2\}$ is precisely the minimal number of modes to ensure $\mathcal{M}^h_{Q}=\mathcal{V}^h$.

The above discussions extend to 3D readily. 
\begin{description}
	\item[1.] $\forall Q\in\mathbb{N}, \mathcal{M}_Q^h \subset \mathcal{V}^h$;
	\item[2.] If $Q_1 \le Q_2, Q_1,Q_2 \in \mathbb{N}, \mathcal{M}_{Q_1}^h \subset \mathcal{M}_{Q_2}^h$;
	\item[3.] $\forall Q\geq \min\{n_1n_2,n_2n_3, n_1n_3\}, Q\in\mathbb{N},  \mathcal{M}_Q^h = \mathcal{V}^h$.
\end{description}

We note that the minimal number of modes to ensure $\mathcal{M}_Q^h = \mathcal{V}^h$ in 3D is essentially to find a best rank-$r$ approximation to order-3 tensor, which is an open mathematical problem. An upper bound is given in property 3.

\section{Convergence for 2D canonical decomposition method at increasing number of modes}
\label{Appdix:2DConverge}

\begin{theorem}
For 2D case, $\mathcal{M}_Q^h$ and $\mathcal{V}^h$ are defined in Eq.(\ref{eq:2DMS}) and Eq.(\ref{eq:2DFEMVh}), respectively. We have
\begin{equation}
    \forall Q \geq \min(n_1,n_2), Q \in \mathbb{N}, \mathcal{M}_Q^h=\mathcal{V}^h.
    \label{eq:2DMconvergeToV}
\end{equation}
\end{theorem}
\begin{proof}
For any interpolation function
\begin{equation}
    u^{h,FEM}=\sum_{i=1}^{n_1}\sum_{j=1}^{n_2} N_i(x)N_j(y) u_{i,j}
\end{equation}
in the set $\mathcal{V}^h$, we write the nodal values in the form of matrix,
\begin{equation}
    \bm{U}=\left[
    \begin{array}{cccc}
        u_{1,1} & u_{1,2} & \cdots & u_{1,n_2} \\
        u_{2,1} & u_{2,2} & \cdots & u_{2,n_2} \\
        \vdots  & \vdots  &        & \vdots \\
        u_{n_1,1} & u_{n_1,2} & \cdots & u_{n_1,n_2}
    \end{array}\right].
\end{equation}
According to the singular value decomposition (SVD), $\bm{U}$ is represented by
\begin{equation}
    \bm{U}=\sum_{q=1}^{rank(\bm{U})} \sigma^{(q)} \bm{w}^{(q)} \otimes \bm{v}^{(q)}, \sigma^{(1)}\geq\sigma^{(2)}\geq\cdots\geq\sigma^{(rank(\bm{U}))}>0,
\end{equation}
where $\bm{w}^{(q)}$ is the $n_1$-dimensional vector, and $\bm{v}^{(q)}$ is the $n_2$-dimensional vector. Thus $u^{h,FEM}$ is rewritten in the form of the separation of variables, i.e.,
\begin{equation}
    u^{h,FEM}=\sum_{i=1}^{n_1} \sum_{j=1}^{n_2} N_i(x) N_j(y) \left(\sum_{q=1}^{rank(\bm{U})}\sigma^{(q)} w_i^{(q)} v_j^{(q)} \right).
\end{equation}
So we have $u^{h,FEM}\in\mathcal{M}_Q^h$, if $Q\geq\min\{n_1,n_2\}\geq rank(\bm{U})$. Combining  with $\mathcal{M}_Q^h\subset\mathcal{V}^h$, we obtain Eq. (\ref{eq:2DMconvergeToV}).
\end{proof}

We remark that SVD tells us the minimal number of modes to reproduce the FE solution, i.e. $\min\{n_1,n_2\}$.

\section{1D HiDeNN Formulation}
\label{sec:1DHiDeNN}

In standard 1D FEM, the computational domain $\Omega$ is discretized by a grid with $n$ nodes and the shape function associated with an internal node $x_I$ is
\begin{equation}
N_I(x)=\left\{
\begin{array}{cc}
\dfrac{x-x_{I-1}}{x_I-x_{I-1}},     & x_{I-1}\leq x \leq x_I, \\
\dfrac{x_{I+1}-x}{x_{I+1}-x_{I}},     & x_I \leq x \leq x_{I+1}, \\
0,   & else where,
\end{array}
\right.
\label{eq:linearshapefunction}
\end{equation}
where $x_{I-1}$ and $x_{I+1}$ are the two neighbor points of the node $\:x_{I}$ from the left side and right side, respectively.

We rewrite $N_I(x)$ in a DNN format consists of weights, biases, and activation functions. Considering the shape function is a piecewise linear function, the activation function is selected as ReLU function, i.e., $\mathcal{A}_{1}=\max(0,x)$. Fig. \ref{fig:1D_DNN_representation}(a) shows the DNN representation of the linear shape function. The corresponding formula is 
\begin{eqnarray}
\mathscr{N}_{I}(x;\bm{W}, \bm{b},\:\bm{\mathcal{A}})
&=& W_{11}^{l=4}\mathcal{A}_{1}\left( W_{11}^{l=3} \mathcal{A}_{1} \left( W_{11}^{l=2}x+b_{1}^{l=2} \right) +b_{1}^{l=3} \right) \\ \nonumber
&&+ W_{21}^{l=4}\mathcal{A}_{1} \left( W_{22}^{l=3} \mathcal{A}_{1} \left( W_{12}^{l=2}x+b_{2}^{l=2} \right) +b_{2}^{l=3} \right) +b_{1}^{l=4}\\ \nonumber
&=& \mathcal{A}_{1}\left( \dfrac{-1}{x_I-x_{I-1}} \mathcal{A}_{1} \left( -x+x_I \right) +1 \right) + \mathcal{A}_{1} \left( \dfrac{-1}{x_{I+1}-x_I} \mathcal{A}_{1} \left( x-x_I \right) +1 \right) -1, \nonumber
\end{eqnarray}
where $\bm{W}=[W_{11}^{l=2},W_{12}^{l=2},W_{11}^{l=3},W_{22}^{l=3},W_{11}^{l=4},W_{21}^{l=4}]$, and $\bm{b}=[b_{1}^{l=2},b_{2}^{l=2},b_{1}^{l=3},b_{2}^{l=3},b_{1}^{l=4}]$ are the weights and biases of the connected neurons. Note that all the weights and biases are functions of nodal coordinates. The formula can be rewritten as the form of
\begin{equation}
    \mathcal{N}_I(\bm{x};\bm{x}_I^*,\bm{\mathcal{A}}),
\end{equation}
where $\bm{x}_I^*$ denotes the vector that represents the neighbor nodes of node $\bm{x}_I$ involved in $N_I(\bm{x})$. For 1D linear shape function, it should be $\bm{x}_I^*=[x_{I-1},\:x_{I},\:x_{I+1}]$. For the sake of clarity, one more layer is added to introduce the nodal value $u_I$, i.e., the formula becomes
\begin{eqnarray}
\mathscr{u}_I^{h}&=&\mathscr{N}_{I} (x;\:\bm{W},\bm{b},\:\bm{\mathcal{A}})\mathscr{u}_{I}=\mathscr{N}_{I} (x;\:\bm{x}_I^*,\:\bm{\mathcal{A}})\mathscr{u}_{I}; \mbox{ no summation on } {I}\\ \nonumber
&=& \mathcal{A}_{0}\left(\mathcal{A}_{1}\left( \dfrac{-1}{x_I-x_{I-1}} \mathcal{A}_{1} \left( -x+x_I \right) +1 \right) -0.5\right) \mathscr{u}_{I} \\ \nonumber
&&+ \mathcal{A}_{0}\left(\mathcal{A}_{1} \left( \dfrac{-1}{x_{I+1}-x_I} \mathcal{A}_{1} \left( x-x_I \right) +1 \right) -0.5\right) \mathscr{u}_{I},
\end{eqnarray}
where $\mathscr{u}_I^{h}$ and $\mathscr{u}_I$ are the interpolated displacement and nodal displacement at node $x_I$, $\bm{\mathcal{A}}=[\mathcal{A}_{0},\:\mathcal{A}_{1}]$ are the activation functions used for the construction of the DNN approximation. $\mathcal{A}_{0}(x)=x$ is an identical function. Fig. \ref{fig:1D_DNN_representation}(b) gives the DNN representation of the interpolation of the nodal displacement at node $x_I$.

\begin{figure}[h]
	\centering
	\subfigure[DNN-based 1D shape function]{\includegraphics[width=0.42\textwidth]{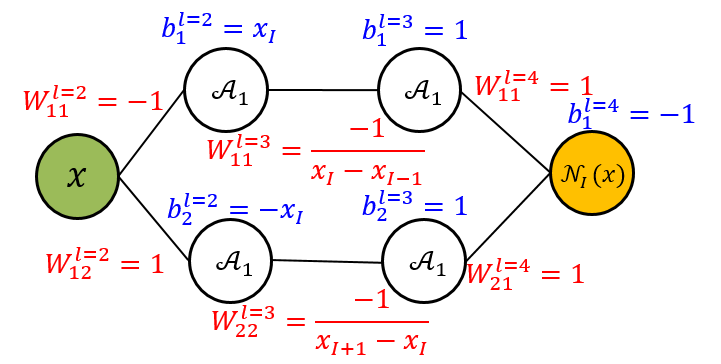}}
	\hspace{0.1in}
	\subfigure[DNN-based 1D interpolation function ]{\includegraphics[width=0.5\textwidth]{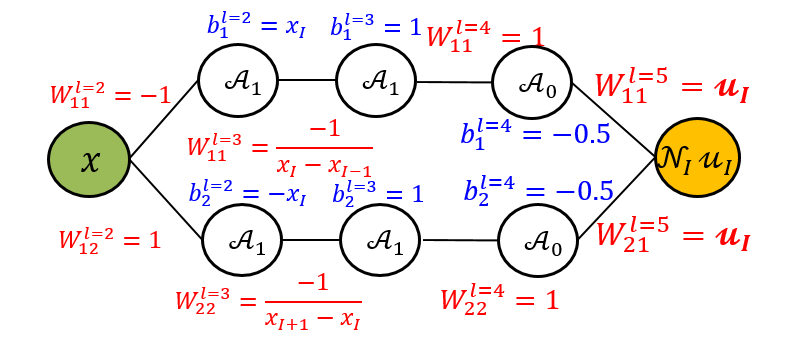}}
	\caption{Deep neural network (DNN) representation of the 1D global shape function and interpolation function.}
	\label{fig:1D_DNN_representation}
\end{figure}

Once the shape function with nodal value for an arbitrary node $x_I$ is constructed, the interpolation is obtained by assembling all DNNs, i.e.,
\begin{equation}
\mathscr{u}^{h}(x)=\sum_{I=1}^{n}\mathscr{N}_{I} (x;\:\bm{x}_I^*,\:\bm{\mathcal{A}})\mathscr{u}_{I}.
\end{equation}
Compared with classical FEM, nodal positions are introduced as additional DoFs in the optimization for HiDeNN, which increases both the local and global accuracy of the interpolants.

Reference \cite{zhang2021hierarchical} also presented the DNN representation of various rational functions including Lagrange polynomials, B-spline, Reproducing Kernel Particle Method (RKPM), NURBS, Isogeometric analysis (IGA), etc., and multidimensional shape functions.

\section{Space separated PGD}
When the domain is not intrinsically separable, fully separated representation can not be applied directly. 
\cite{gonzalez2010recent} immersed the non-separable domain onto a fully separable one. \cite{ghnatios2019advanced} used geometrical mapping to deal with layered domain, where interfaces are not planar. Now we combine representation of separated variables with FE mapping to deal with more complex geometrical cases. In section 4.1, we define the mapping to a regular parameter space by means of FE geometrical mapping. Then section 4.2 addresses several solution schemes. Finally, we take a 2D Poisson problem as illustration for the whole solution procedure in section 4.3.

\subsection{Mesh mapping and recovering for irregular domains}

\begin{figure}[h]
	\centering
	\includegraphics[width=0.9\textwidth]{PolarTransform.png}
	\caption{A quarter of ring is transformed in to a rectangular by the polar coordinates transformation.}
	\label{fig:PolarTrans}
\end{figure}

The main idea is to map original irregular domain $\Omega_{(\bm{x})}$ to a regular one $\tilde{\Omega}_{(\tilde{\bm{x}})}$, and then apply the separated representation. Fig, \ref{fig:PolarTrans} illustrates a simple example. A quarter of ring becomes a rectangular through polar transformation. Then the new representation of separated variables is shown as below,
\begin{equation}
    u^h=\sum_{q=1}^Q \tilde{X}^{(q)}(\theta)\tilde{Y}^{(q)}(r),
\end{equation}
where $\theta$ and $r$ are the functions of space coordinates $x,y$. 
\begin{figure}[h]
	\centering
	\includegraphics[width=0.9\textwidth]{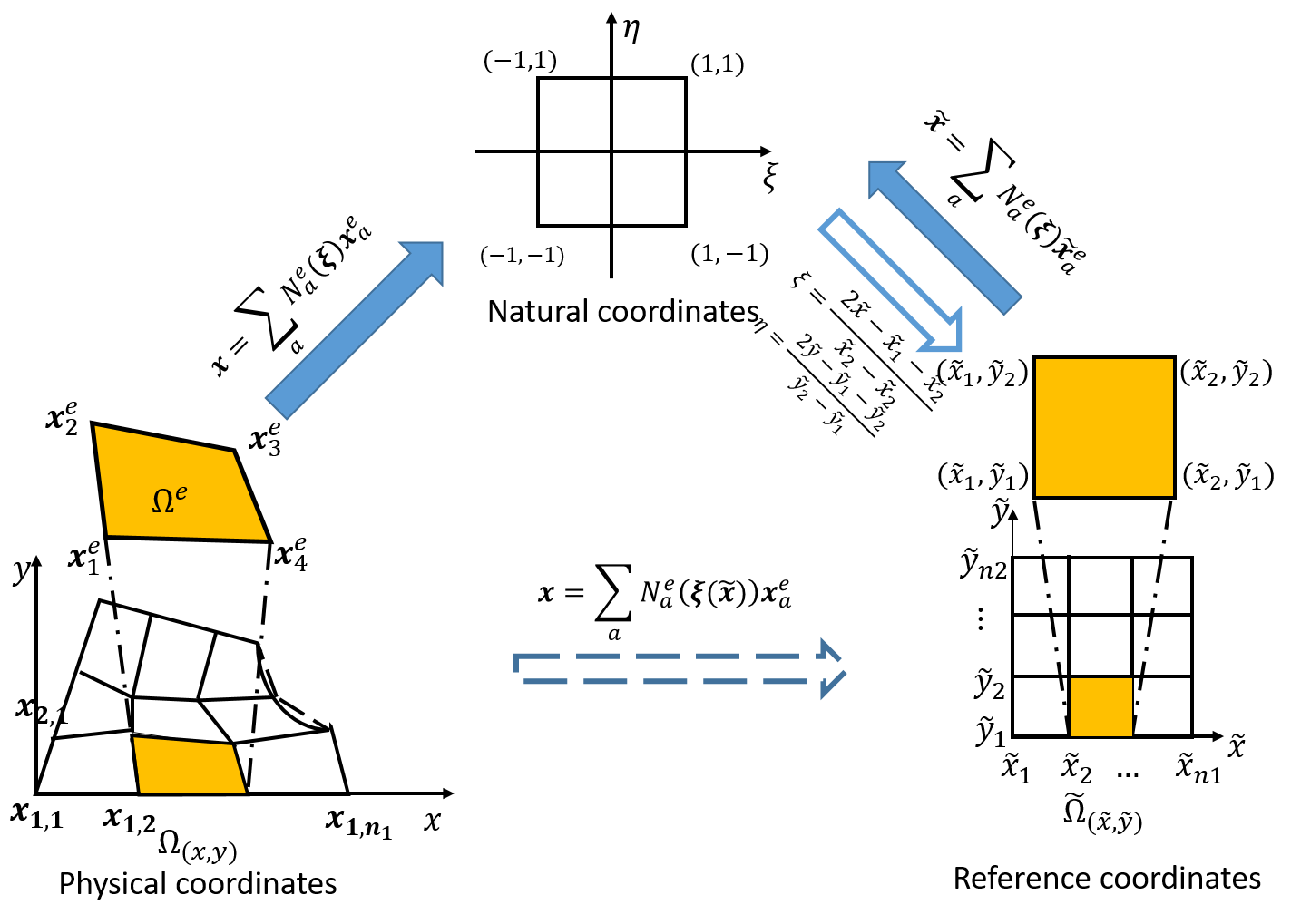}
	\caption{Illustration for the geometrical mapping. The irregular domain with irregular mesh is related to a regular domain with regular mesh by a 2-step mapping.}
	\label{fig:FETrans}
\end{figure}
By virtue of parametric transformation in FEM, we propose a general way to define this mapping as illustrated in Fig \ref{fig:FETrans}. We present a FE mesh over the 2D irregular computational domain $\Omega_{\bm{x}}$ first with nodes $\bm{x}_{i,j}, i=1,2,\cdots,n_1, j=1,2,\cdots,n_2$. Then we define a mapping to its corresponding lattice $(i,j)$. $\tilde{\bm{x}}$ is the coordinates of the transformed domain $\tilde{\Omega}_{\tilde{\bm{x}}}$. The mapping consist of two steps:

1. Mapping each element to a square or cubic

The first mapping is the classical parametric mapping in FEM. We make a change of coordinates which maps the 4-node element into a square $[-1,1]^2$ for 2D or maps the 8-node element into a cubic $[-1,1]^3$ for 3D. The coordinates of a point $\bm{\xi}$ in the square is related to the physical coordinates of a point $\bm{x}$ in the element by mappings of the form
\begin{equation}
    \bm{x}=\sum_{a=1}^{n_e} N_a^e(\bm{\xi}) \bm{x}_a^e
\end{equation}
where $n_e$ is the number of nodes of the element ($n_e=4$ for 2D and $n_e=8$ for 3D), $\bm{x}_a^e$ is the coordinates of the $a$-th node of the element, and $N_a^e(\bm{\xi})$ is the corresponding shape function.
$\bm{\xi}$ is called as natural coordinates.

2. Mapping the square to a lattice

For the sake of separated representation, we define the second mapping to assemble the square into a lattice. The transformed formula is
\begin{equation}
    \tilde{\bm{x}}=\sum_{a=1}^{n_e} N_a^e(\bm{\xi}) \tilde{\bm{x}}_a^e
\end{equation}
and the inverse transformation is
\begin{eqnarray}
    \xi&=&\dfrac{2\tilde{x}-\tilde{x}^e_1-\tilde{x}^e_2}{\tilde{x}^e_2-\tilde{x}^e_1} \\
    \eta&=&\dfrac{2\tilde{y}-\tilde{y}^e_1-\tilde{y}^e_2}{\tilde{y}^e_2-\tilde{y}^e_1} \\
    \zeta&=&\dfrac{2\tilde{z}-\tilde{z}^e_1-\tilde{z}^e_2}{\tilde{z}^e_2-\tilde{z}^e_1}.
\end{eqnarray}
The final transformed domain $\tilde{\Omega}_{\tilde{\bm{x}}}$ is called as reference domain, which is a regular domain with a regular mesh $[\tilde{x}_1, \tilde{x}_2, \cdots, \tilde{x}_{n_1}]\times[\tilde{y}_1, \tilde{y}_2, \cdots, \tilde{y}_{n_2}]\times[\tilde{z}_1, \tilde{z}_2, \cdots, \tilde{z}_{n_3}]$. $\tilde{x}^e_1,\tilde{x}^e_2,\tilde{y}^e_1,\tilde{y}^e_2,\tilde{z}^e_1,\tilde{z}^e_2$ are the coordinates of the element $[\tilde{x}^e_1,\tilde{x}^e_2]\times[\tilde{y}^e_1,\tilde{y}^e_2]\times[\tilde{z}^e_1,\tilde{z}^e_2]$ in the reference domain. For convenience, we might take the mesh in the reference domain as a lattice corresponding to the index of the nodes in the physical domain, i.e., $\tilde{x}_i=i,i=1,2,\cdots,n_1, \tilde{y}_j=j,j=1,2,\cdots,n_2, \tilde{z}_k=k,k=1,2,\cdots,n_3$. 

The whole mapping is defined as below,
\begin{equation}
    \bm{x}=\sum_{a=1}^{n_e} N_a^e(\bm{\xi}(\tilde{\bm{x}})) \bm{x}_a^e.
    \label{eq:WholeMap}
\end{equation}
Then separation of spatial variables is applicable to the reference domain, i.e., the interpolation function set is
\begin{equation}
    \tilde{\mathcal{M}}^h_Q=\left\{u^h\Bigg |u^h=\sum_{q=1}^Q \tilde{X}(\tilde{\bm{x}})\tilde{Y}(\tilde{y})\tilde{Z}(\tilde{z})=\sum_{q=1}^Q \left( \sum_{i=1}^{n_1}N_i(\tilde{x})\beta_i^{(q)} \right) \left( \sum_{j=1}^{n_2}N_j(\tilde{y})\gamma_j^{(q)} \right) \left( \sum_{k=1}^{n_3}N_k(\tilde{z})\theta_k^{(q)} \right)\right\},
\end{equation}
where $N_i(\tilde{x}),N_j(\tilde{y}),N_i(\tilde{z})$ are shape functions, and $\beta_i^{(q)}, \gamma_j^{(q)}, \theta_k^{(q)}$ are the corresponding coefficients of the $q$-th mode.
Note that $\tilde{x},\tilde{y},\tilde{z}$ are the functions of physical coordinates $x,y,z$.

\subsection{Solution schemes}

There are different solution schemes to solve the problem. A straight way is to find the solution by minimizing the variational formula with the given number of modes $Q^*$ directly, i.e.,
\begin{equation}
    u^h=\argmin_{u^h\in \tilde{\mathcal{M}}^h_{Q=Q^*}} \Pi(u^h(\bm{x};\bm{\beta}^{(q)},\bm{\gamma}^{(q)},\bm{\theta}^{(q)})).
\end{equation}
All the parameters $\beta$ are solved at the same time. 

Yet this global optimization might be expensive, so we might borrow the idea from PGD \cite{ammar2006new,chinesta2013proper}, i.e., incremental solution scheme. More precisely, the solution scheme is
\begin{equation}
\begin{array}{lll}
\text{The first mode:} &u^{PGD,(1)}=\arg \min_{\Delta u \in \tilde{\mathcal{M}}^h_{Q=1}} \Pi[u^0+\Delta u];\\
\text{For $m>1$, the $m$-th mode:} &\Delta u^{(m)}=\arg \min_{\Delta u \in \tilde{\mathcal{M}}^h_{Q=1}} \Pi[u^{PGD,(m-1)}+\Delta u], \\
\text{The PGD solution with $m$ modes:} &u^{PGD,(m)}=u^{PGD,(m-1)}+\Delta u^{(m)},
\end{array}
\end{equation}
We remarked it is also possible to solve several modes simultaneously in one incremental step.

In general, the initial guess $u_0$ is set to be zero. When dealing with the boundary conditions, $u^0$ can be arbitrary continuous functions satisfying boundary conditions. We can also set an appropriate initial guess to improve the efficiency. 

These two solution schemes have their advantages and disadvantages. With the same number of modes $Q^*$, we solve the modes one by one using the latter solution scheme while optimize all the modes together using the previous one. Thus, we have $\Pi(u^{CD})<\Pi(u^{PGD})$ and then obtain
\begin{equation}
    \| u^{CD}-u^{exact} \|_E \leq \| u^{PGD}-u^{exact} \|_E.
\end{equation}
This indicates the solution of the previous solution scheme might be better than that of the incremental way. However, the previous one might cost more.

\subsection{Illustrating the solution procedure: the 2D Poisson problem}
For the sake of simplicity and without loss of generality, we consider 2D Poisson problem with incremental solution scheme for illustration,
\begin{equation}
\left\{\begin{array}{l}
    \nabla^2 \bm{u}(x,y)+b(x,y)=0 \text{ in } \Omega_{(x,y)} \subset \mathbb{R}^2, \\
    \bm{u}|_{\partial \Omega}=\bm{0}.
\end{array}\right.
\label{eq:PoissonEq}
\end{equation}
(\ref{eq:PoissonEq}) is solved in the irregular domain $\Omega_{(x,y)}$ with homogeneous boundary conditions. 

The solution is assumed in the form of (\ref{eq:2DMS}). Then we solve it with incremental solution scheme. Let previous $q-1$ modes solved. The $q$-th mode is obtained by
\begin{equation}
    \Delta u^{(q)}=\argmin_{\Delta u \in \tilde{\mathcal{M}}^h_{Q=1}} \Pi[u^{PGD,(q-1)}+\Delta u^{(q)}],
\end{equation}
where $u^{PGD,(q-1)}$ is the sum of the previous $q-1$ modes.
We rewrite the interpolation function in the following matrix form
\begin{equation}
    u^{PGD,(q)}=u^{PGD,(q-1)}+\Delta u^{(q)}, \Delta u^{(q)}=\left((\bm{\beta}^{(q)})^T \bm{N}^{\beta}(\tilde{x})\right)\left((\bm{\gamma}^{(q)})^T \bm{N}^{\gamma}(\tilde{y})\right),
    \label{eq:PGDway}
\end{equation}
where $\bm{\beta}^{(q)}, \bm{\gamma}^{(q)}$ are the coefficient vector, and $\bm{N}^{\beta}(\tilde{x}),\bm{N}^{\gamma}(\tilde{y})$ denotes the vector containing shape functions.

Substituting (\ref{eq:PGDway}) into the variational formula (\ref{eq:PoissonVar}), we have
\begin{eqnarray}
    \Pi(u^{PGD,(q)})&=&\dfrac{1}{2}\int_{\Omega_{(x,y)}} \left(\nabla (\Delta u^{(q)})\right)^2 \mathrm{d}x\mathrm{d}y+\int_{\Omega_{(x,y)}} \left(\nabla (\Delta u^{(q)})\right)\left(\nabla (u^{PGD,(q-1)})\right) \mathrm{d}x\mathrm{d}y \\ \nonumber
    &&- \int_{\Omega_{(x,y)}} u(x,y)b(x,y)\mathrm{d}x\mathrm{d}y + \Pi(u^{PGD,(q-1)}).
\end{eqnarray}

The quadratic term with respect to $\bm{\beta}^{(q)}, \bm{\gamma}^{(q)}$ in the variational formula is given by
\begin{eqnarray}
    \int_{\Omega_{(x,y)}} \left(\nabla (\Delta u^{(q)})\right)^2 \mathrm{d}x\mathrm{d}y &=& \int_{\tilde{\Omega}_{(\tilde{x},\tilde{y})}} \left(\nabla_{(\tilde{x},\tilde{y})} (\Delta u^{(q)})\right)^T \bm{J}^{-T} \bm{J}^{-1} \nabla_{(\tilde{x},\tilde{y})} (\Delta u^{(q)}) \mathrm{det}(\bm{J}) \mathrm{d}\tilde{x}\mathrm{d}\tilde{y}.
    \label{eq:Integration}
\end{eqnarray}
with the Jacobi matrix
\begin{equation}
    \bm{J}=\dfrac{\partial(x,y)}{\partial(\tilde{x},\tilde{y})}.
\end{equation}
The gradient of $\Delta u^{(q)}$ is
\begin{equation}
    \nabla_{(\tilde{x},\tilde{y})}\Delta u^{(q)}=
    \left[\begin{array}{c}
       \dfrac{\partial}{\partial \tilde{x}}\\
       \dfrac{\partial}{\partial \tilde{y}}
    \end{array}\right]
    \Delta u^{(q)}=
    \left[\begin{array}{c}
       (\bm{\beta}^{(q)})^T \dfrac{d\bm{N}^{\beta}(\tilde{x})}{d\tilde{x}}(\bm{N}^{\gamma}(\tilde{y}))^T \gamma^{(q)}\\
       (\bm{\beta}^{(q)})^T \bm{N}^{\beta}(\tilde{x})(\dfrac{d\bm{N}^{\gamma}(\tilde{y})}{d\tilde{y}})^T \gamma^{(q)}
    \end{array}\right].
\end{equation}
Note that the expression is a nonlinear algebraic system. It is hard to solve it directly, so we use the alternating direction strategy as below:

In each iteration step,

1. Fix $\gamma$ and solve $\beta$

The quadratic term becomes
\begin{equation}
    \int_{\Omega_{(x,y)}} \left(\nabla (\Delta u^{(q)})\right)^2 \mathrm{d}x\mathrm{d}y=\int_{\tilde{\Omega}_{(\tilde{x},\tilde{y})}} \beta^{(q),T} \bm{B}^{\beta,T}(\tilde{x},\tilde{y})\bm{B}^{\beta}(\tilde{x},\tilde{y})\beta^{(q)} \mathrm{det}(\bm{J}) \mathrm{d}\tilde{x}\mathrm{d}\tilde{y}
\end{equation}
with
\begin{equation}
    \bm{B}^{\beta}(\tilde{x},\tilde{y})=\bm{J}^{-1}
    \left[\begin{array}{c}
       (\bm{\gamma}^{(q)})^T\bm{N}^{\gamma}(\tilde{y}) (\dfrac{d\bm{N}^{\beta}(\tilde{x})}{d\tilde{x}})^T\\
       (\bm{\gamma}^{(q)})^T\dfrac{d\bm{N}^{\gamma}(\tilde{y})}{d\tilde{y}} (\bm{N}^{\beta}(\tilde{x}))^T
    \end{array}\right].
\end{equation}
The stiffness matrix for $\beta^{(q)}$ is 
\begin{equation}
    \bm{K}^{\beta}=\int_{\tilde{\Omega}_{(\tilde{x},\tilde{y})}} \bm{B}^{\beta,T}(\tilde{x},\tilde{y})\bm{B}^{\beta}(\tilde{x},\tilde{y}) \mathrm{det}(\bm{J}) \mathrm{d}\tilde{x}\mathrm{d}\tilde{y}.
\end{equation}

2. Fix $\beta$ and solve $\gamma$

The quadratic term becomes
\begin{equation}
    \int_{\Omega_{(x,y)}} \left(\nabla (\Delta u^{(q)})\right)^2 \mathrm{d}x\mathrm{d}y=\int_{\tilde{\Omega}_{(\tilde{x},\tilde{y})}} \gamma^{(q),T} \bm{B}^{\gamma,T}(\tilde{x},\tilde{y})\bm{B}^{\gamma}(\tilde{x},\tilde{y})\gamma^{(q)} \mathrm{det}(\bm{J}) \mathrm{d}\tilde{x}\mathrm{d}\tilde{y}
\end{equation}
with
\begin{equation}
    \bm{B}^{\gamma}(\tilde{x},\tilde{y})=\bm{J}^{-1}
    \left[\begin{array}{c}
       (\bm{\beta}^{(q)})^T \dfrac{d\bm{N}^{\beta}(\tilde{x})}{d\tilde{x}}(\bm{N}^{\gamma}(\tilde{y}))^T\\
       (\bm{\beta}^{(q)})^T \bm{N}^{\beta}(\tilde{x})(\dfrac{d\bm{N}^{\gamma}(\tilde{y})}{d\tilde{y}})^T
    \end{array}\right].
\end{equation}
The stiffness matrix for $\gamma^{(q)}$ is 
\begin{equation}
    \bm{K}^{\gamma}=\int_{\tilde{\Omega}_{(\tilde{x},\tilde{y})}} \bm{B}^{\gamma,T}(\tilde{x},\tilde{y})\bm{B}^{\gamma}(\tilde{x},\tilde{y}) \mathrm{det}(\bm{J}) \mathrm{d}\tilde{x}\mathrm{d}\tilde{y}.
\end{equation}

If we present a regular mesh over a regular domain $\Omega_{(x,y)}$, the mapping is a linear transformation for coordinates. Let one element of the regular mesh $[x^e_1,x^e_2]\times[y^e_1,y^e_2]$. The mapping (\ref{eq:WholeMap}) reduces to
\begin{equation}
    x=\dfrac{x^e_2-x^e_1}{\tilde{x}^e_2-\tilde{x}^e_1}(\tilde{x}-\tilde{x}^e_1)+x^e_1,
    y=\dfrac{y^e_2-y^e_1}{\tilde{y}^e_2-\tilde{y}^e_1}(\tilde{y}-\tilde{y}^e_1)+y^e_1.
\end{equation}
$\bm{J}$ reduces to a diagonal matrix $\mathrm{diag}(\dfrac{x^e_2-x^e_1}{\tilde{x}^e_2-\tilde{x}^e_1},\dfrac{y^e_2-y^e_1}{\tilde{y}^e_2-\tilde{y}^e_1})$. Thus we have
\begin{equation}
    \bm{J}^{-T}\bm{J}^{-1}\mathrm{det}(J)=\dfrac{x^e_2-x^e_1}{\tilde{x}^e_2-\tilde{x}^e_1}\dfrac{y^e_2-y^e_1}{\tilde{y}^e_2-\tilde{y}^e_1},
\end{equation}
which is constant in each element and separated representation. Thus this method degenerates to the classical PGD. For irregular mesh, $\bm{J}^{-T}\bm{J}^{-1}\mathrm{det}(J)$ in (\ref{eq:Integration}) is the function of $\bm{x}$ and mostly non-separated representation. In the numerical implementation, we usually approximate it by a separated form using SVD technique, i.e.,
\begin{equation}
    \bm{J}^{-T}\bm{J}^{-1}\mathrm{det}(J)\approx 
    \left[\begin{array}{cc}
        \sum_a \phi_{11}^{(a)}(\tilde{x})\psi_{11}^{(a)}(\tilde{y}) & \sum_a\phi_{12}^{(a)}(\tilde{x})\psi_{12}^{(a)}(\tilde{y}) \\
        \sum_a\phi_{12}^{(a)}(\tilde{x})\psi_{12}^{(a)}(\tilde{y}) & \sum_a\phi_{22}^{(a)}(\tilde{x})\psi_{22}^{(a)}(\tilde{y})
    \end{array}\right].
\end{equation}
This converts the 2D integration (\ref{eq:Integration}) to the product of 1D integration along different directions ($\tilde{x}$ and $\tilde{y}$ directions in the reference domain), which might reduce the computational cost for integration.

\subsection{Numerical exmples}

\begin{figure}[htbp]
    \centering
    \subfigure[Physical domain ]{\includegraphics[width=3in]{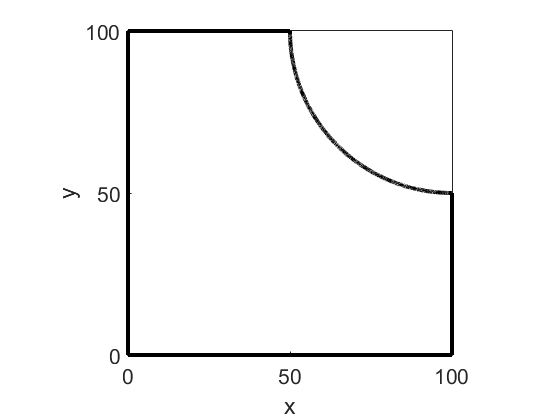}}
    \subfigure[FE mesh]{\includegraphics[width=3in]{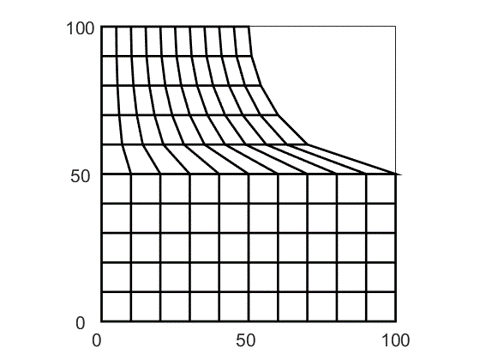}}
    \subfigure[Reference solution ]{\includegraphics[width=3in]{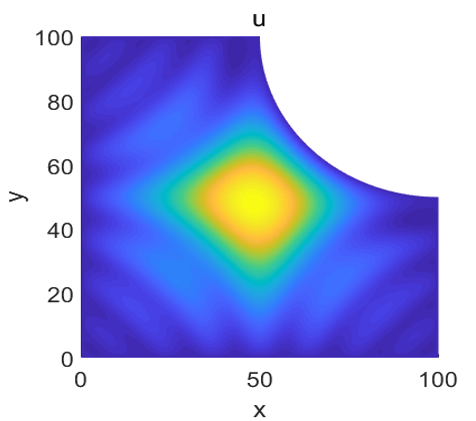}}
	\caption{Geometry and mesh for the plate. (a) Model of the plate. (b) $100$ elements mesh as an illustration. (c) Reference solution with $4.90\times 10^7$ elements.}
	\label{fig:PGDEg_Corner}
\end{figure}

\begin{figure}[h]
    \centering
    \subfigure[Physical domain ]{\includegraphics[width=2in]{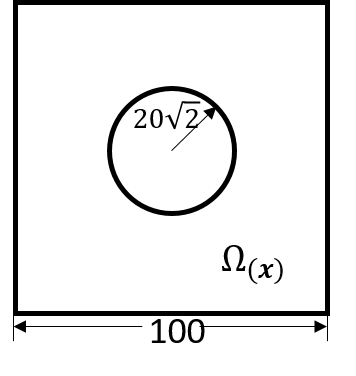}}
    \subfigure[FE mesh]{\includegraphics[width=3in]{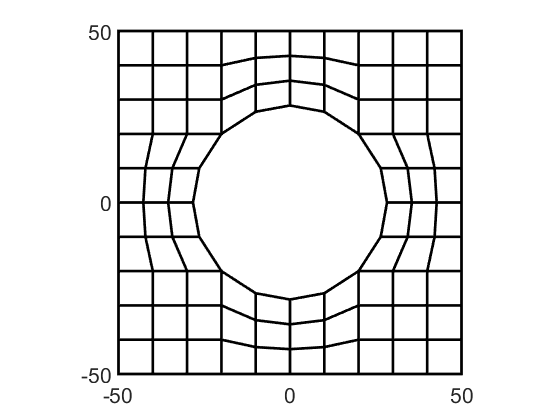}}
    \subfigure[Reference mesh]{\includegraphics[width=2.2in]{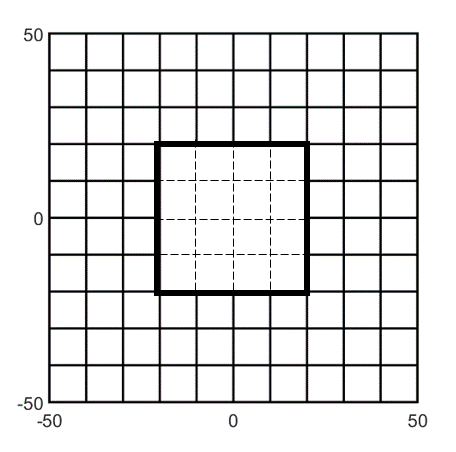}}
    \subfigure[Reference solution ]{\includegraphics[width=3in]{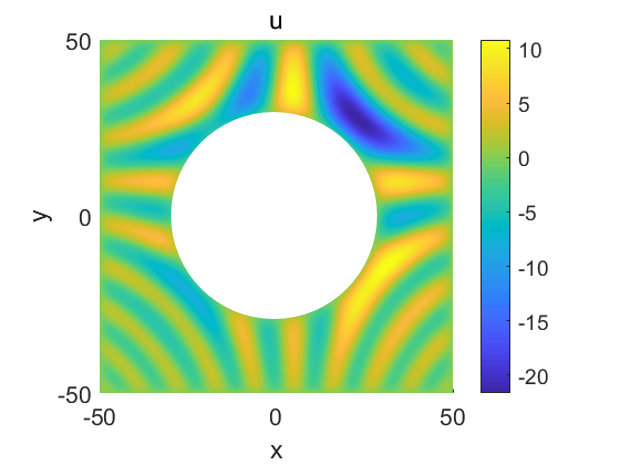}}
	\caption{Geometry and mesh for the plate with one hole. (a) Model of the plate with one hole. (b) Discretization by a $84$ elements mesh. (c) Reference mesh corresponding to $84$ elements.}
	\label{fig:PGDEg_Hole}
\end{figure}
